\documentclass[11pt]{article}
\usepackage{tocloft}
\usepackage{fancyhdr}

\usepackage{tgschola}
\usepackage{macro}

\date{}

\newcommand{\One}{{\mathrm{\uppercase\expandafter{\romannumeral 1}}}}
\newcommand{\Two}{{\mathrm{\uppercase\expandafter{\romannumeral 2}}}}
\newcommand{\rhodn}{{\rho_{\pmb{d},n}}}

\newcommand{\rhoDN}[2]{{\rho_{\pmb{#1},#2}}}

\newcommand{\XXX}{\mathfrak{X}}
\newcommand{\Qp}{\Q_p}

\newcommand{\cZ}{\mathcal{Z}}

\newcommand{\Int}{\mathrm{int}}
\newcommand{\Sum}{\mathrm{sum}}

\title{Codimension two cycles in Iwasawa theory and tensor product of Hida families}

\author{Antonio Lei, Bharathwaj Palvannan
}

\newcommand{\Addresses}{{
  \bigskip

  Antonio Lei, \textsc{D\'epartement de math\'ematiques et de statistique, Universit\'e Laval, Pavillon Alexandre-Vachon, 1045 avenue de la M\'edecine, Qu\'ebec, Canada, G1V 0A6.} \par\nopagebreak \textit{E-mail address}: \texttt{antonio.lei@mat.ulaval.ca}

  \medskip

  Bharathwaj Palvannan, \textsc{National Center for Theoretical Sciences,
No. 1 Sec. 4, Roosevelt Rd., National Taiwan University,
Taipei, 106, Taiwan} \par\nopagebreak
  \textit{E-mail address}:  \texttt{pbharathwaj@ncts.tw}
}}

\begin{document}
\maketitle

\begin{abstract}
The purpose of this paper is to build on results in {\it{higher codimension Iwasawa theory}}.  The setting of our results involves Galois representations arising as cyclotomic twist deformations associated to (i) the tensor product of two cuspidal Hida families $F$ and $G$, and (ii) the tensor product of three cuspidal Hida families $F$, $G$ and $H$. On the analytic side, we consider (i) a pair of $3$-variable Rankin-Selberg $p$-adic $L$-functions constructed by Hida and (ii) a balanced $4$-variable $p$-adic $L$-function (due to Hsieh and Yamana) and an unbalanced  $4$-variable $p$-adic $L$-function (whose existence is currently conjectural). In each of these setups, when the two $p$-adic $L$-functions generate a height two ideal in the corresponding deformation ring, we use codimension two cycles of that ring to relate them to a pair of pseudo-null modules.
\end{abstract}

\tableofcontents

\section{Introduction}
\thispagestyle{fancy}
The topic of \textit{higher codimension Iwasawa theory} was initiated in the work of \cite{bleher2015higher}. The underlying objective is to associate analytic invariants to pseudo-null modules in Iwasawa theory. Let us first introduce the necessary notations to state our results precisely.

Let $p \geq 5$ be a fixed prime number. Let $F$, $G$ and $H$ be three ordinary cuspidal Hida families defined over $\Q$. Let $R_F$, $R_G$ and $R_H$ denote the integral closures of the irreducible components of the ordinary primitive cuspidal Hecke algebras through $F$,  $G$ and $H$ respectively. The rings $\R_F$, $\R_G$ and $R_H$ turn out to be finite integral extensions of the power series rings $\Z_p[[x_F]]$, $\Z_p[[x_G]]$  and $\Z_p[[x_H]]$ respectively. The variables $x_F$,  $x_G$ and $x_H$ are the ``weight variables'' for the Hida families $F$, $G$ and $H$ respectively. Throughout this paper, we will impose the following hypotheses on $F$, $G$ and $H$:

\pagestyle{plain}
\begin{enumerate}[style=sameline, style=sameline, align=left,label=\textsc{IRR} --- , ref=\textsc{IRR}]
\item\label{hyp:residual-irr} The residual representations associated to $F$, $G$ and $H$ are absolutely irreducible.
\end{enumerate}
\begin{enumerate}[style=sameline, style=sameline, align=left,label=\textsc{$p$-DIS}  --- , ref=\textsc{$p$-DIS}]
\item \label{hyp:p-distinguished} The restrictions, to the decomposition subgroup $\Gal{\overline{\Q}_p}{\Q_p}$ at $p$, of the residual representations associated to $F$, $G$ and $H$ have non-scalar semi-simplifications.
\end{enumerate}

Let $\Sigma$ be a finite set of primes in $\Q$ containing the primes $p$, $\infty$, a finite prime $l_0 \neq p$ and all the primes dividing the tame levels of $F$, $G$ and $H$. Let $\Sigma_0$ denote $\Sigma \setminus \{p\}$. Let $\Q_\Sigma$ denote the maximal extension of $\Q$ unramified outside $\Sigma$.  Let $G_\Sigma$ denote $\Gal{\Q_\Sigma}{\Q}$.   Assume that the integral closures of $\Z_p$ in $\R_F$, $\R_G$ and $R_H$ are all equal (extending scalars if necessary)~to~$O$. Let $\Q_\Cyc$ denote the cyclotomic $\Z_p$-extension of $\Q$. Let $\Gamma_\Cyc$ denote $\Gal{\Q_\Cyc}{\Q}$. Let $\kappa:G_\Sigma \twoheadrightarrow \Gamma_\Cyc \hookrightarrow \Gl_1(O[[\Gamma_\Cyc]])$ denote the tautological character. We have a two-dimensional Galois representation $\rho_F: G_\Sigma \rightarrow \Gl_2(\R_F)$ associated to $F$ (see Hida's work \cite{hida1986galois}) such that $\mathrm{Trace}(\rho_F)\left(\Frob_l\right)$ equals $a_l(F)$ for all primes $l \notin \Sigma$. Here, the element $a_l(F)$ in $R_F$ denotes the $l^{\text{th}}$-Fourier coefficient of $F$. We let $L_F$ denote the free $\R_F$-module  of rank $2$ on which $G_\Sigma$ acts to let us obtain the Galois representation $\rho_F$.  We can similarly define two-dimensional Galois representations $\rho_G: G_\Sigma \rightarrow \Gl_2(\R_G)$ and $ \rho_H: G_\Sigma \rightarrow \Gl_2(\R_H)$  along  with the associated rank two $R_G$-module $L_G$ and the associated rank two $R_H$-module $L_H$. Let $\psi:G_\Sigma \rightarrow O^\times$ denote a finite continuous character.

We will use the symbol $\hotimes$ to denote the completed tensor product over $O$. We will use the symbol $(\cdot)^\vee$ to  denote the Pontryagin dual.

\subsection*{Tensor product of two Hida families}

Let $R_{F,G}$ denote the completed tensor product  $R_F \hotimes R_G$.  The ring $R_{F,G}$ is an integrally closed domain and a finite integral extension of $\Z_p[[x_F,x_G]]$.  We will consider the Iwasawa algebras $\displaystyle O[[\Gamma_\Cyc]]:=\varprojlim_n O[\Gamma_\Cyc / \Gamma_\Cyc^{p^n}]$ and $\displaystyle R_{F,G}[[\Gamma_\Cyc]]:=\varprojlim_n R_{F,G}[\Gamma_\Cyc / \Gamma_\Cyc^{p^n}]$. We have a 4-dimensional Galois representation $\rhoDN{4}{3}:G_\Sigma \rightarrow \Gl_4\left(R_{F,G}[[\Gamma_\Cyc]]\right)$, given by the action of $G_\Sigma$ on $L_{\pmb{4},{3}}:=L_F \ \hotimes \  L_G\  \hotimes \ O[[\Gamma_\Cyc]] (\psi\kappa^{-1})$, a free  $R_{F,G}[[\Gamma_\Cyc]]$-module of rank four. Let $D_{\pmb{4},{3}}$ denote the discrete $R_{F,G}[[\Gamma_\Cyc]]$-module $L_{\rhoDN{4}{3}} \otimes_{R_{F,G}[[\Gamma_\Cyc]]} R_{F,G}[[\Gamma_\Cyc]]^\vee$.

The Galois representation $\rhoDN{4}{3}$ satisfies two distinct Panchishkin conditions (a terminology coined by Greenberg in \cite{greenberg1994iwasawa}) depending on the choice of the dominant Hida family. Corresponding to the two Panchishkin conditions, there are two $3$-variable Rankin-Selberg $p$-adic $L$-functions $L_p(\underline{\mathbf{F}},G,\psi,s)$ and $L_p(F,\underline{\mathbf{G}},\psi,s)$ in this setup. These $p$-adic $L$-functions, constructed\footnote{One must modify the $p$-adic $L$-function constructed in \cite{Hidafourier} by multiplying it with a one-variable $p$-adic $L$-function associated to the 3-dimensional adjoint representation $\Ad^0(\rho_F)$ (see Conjecture 1.0.1 in \cite{hida1996search}). A discussion surrounding the need to introduce this modification for the Iwasawa main conjectures, which is related to the choice of a certain period (N\'eron period versus a period involving the Peterson inner product), is carefully explained in Hida's article \cite[Sections 1 and 6]{hida1996search}.} by Hida \cite{Hidafourier, hida1996search}, are elements of the fraction field of $R_{F,G}[[\Gamma_\Cyc]]$. For these $p$-adic $L$-functions, in addition to the weights of $F$ and $G$, one can also vary the cyclotomic variable. Let $\theta^{\One}_{\pmb{4},{3}}$ and $\theta^{\Two}_{\pmb{4},{3}}$ denote $L_p(\underline{\mathbf{F}},G,\psi,s)$ and  $L_p(F,\underline{\mathbf{G}},\psi,s)$ respectively.

\subsection*{Tensor product of three Hida families}

Let $R_{F,G,H}$ denote the completed tensor product  $R_F \hotimes R_G \hotimes R_H$.  The ring $R_{F,G,H}$ is an integrally closed domain and a finite integral extension of $\Z_p[[x_F,x_G,x_H]]$.  We will also consider the Iwasawa algebra $\displaystyle R_{F,G,H}[[\Gamma_\Cyc]]:=\varprojlim_n R_{F,G,H}[\Gamma_\Cyc / \Gamma_\Cyc^{p^n}]$. We have an 8-dimensional Galois representation $\rhoDN{8}{4}:G_\Sigma \rightarrow \Gl_8\left(R_{F,G,H}[[\Gamma_\Cyc]]\right)$, given by the action of $G_\Sigma$ on $L_{\pmb{8},{4}}:=L_F \ \hotimes \  L_G\  \hotimes L_H \ \hotimes \ O[[\Gamma_\Cyc]] (\psi\kappa^{-1})$, a free  $R_{F,G,H}[[\Gamma_\Cyc]]$-module of rank eight. Let $D_{\pmb{8},{4}}$ denote the discrete $R_{F,G,H}[[\Gamma_\Cyc]]$-module $L_{\pmb{8},{4}} \otimes_{R_{F,G,H}[[\Gamma_\Cyc]]} R_{F,G,H}[[\Gamma_\Cyc]]^\vee$.

The Galois representation $\rhoDN{8}{4}$ satisfies four distinct Panchishkin conditions. The existence of the $4$-variable $p$-adic $L$-functions, predicted by Greenberg,  corresponding to each of these Panchishkin conditions is (currently) conjectural. Depending on the choice of the dominant Hida family, one should conjecturally have three $4$-variable ``unbalanced'' $p$-adic $L$-functions $L_p^{\mathrm{unb}}(\underline{\mathbf{F}},G,H,\psi,s)$, $L_p^{\mathrm{unb}}(F,\underline{\mathbf{G}},H,\psi,s)$ and $L_p^{\mathrm{unb}}(F,G,\underline{\mathbf{H}},\psi,s)$ in the fraction field of $R_{F,G,H}[[\Gamma_\Cyc]]$. One should also conjecturally have a ``balanced'' $4$-variable $p$-adic $L$-function $L_p^{\mathrm{bal}}(F,G,H,\psi,s)$ in the fraction field of $R_{F,G,H}[[\Gamma_\Cyc]]$. For the unbalanced and the balanced $p$-adic $L$-functions, in addition to the weights of $F$, $G$ and $H$, one can also vary the cyclotomic variable. As evidence to the existence of these $p$-adic $L$-functions, Hsieh \cite{hsieh2016hida} has recently constructed both the unbalanced and the balanced $3$-variable $p$-adic $L$-functions (without the cyclotomic variable) in this setup. Hsieh and Yamana have announced the construction of the $4$-variable balanced $p$-adic $L$-function in a forthcoming work \cite{hsiehyamana}. See also related work of Harris-Tilouine \cite{MR1835065}, B\"ocherer-Panchishkin \cite{MR2290585}, Darmon-Rotger \cite{MR3250064} and Greenberg-Seveso \cite{greenberg2015triple}. Let $\theta^{\One}_{\pmb{8},{4}}$ denote $L_p^{\mathrm{unb}}(\underline{\mathbf{F}},G,H,\psi,s)$ and $\theta^{\Two}_{\pmb{8},{4}}$ denote $L_p^{\mathrm{bal}}(F,G,H,\psi,s)$.

\subsection{Main Result in higher codimension Iwasawa theory}

Let $\RRR$ denote an integrally closed Noetherian domain. Let $Z^i(\RRR)$ denote the free abelian group on the set of height $i$ prime ideals of $\RRR$. A finitely generated torsion $\RRR$-module $\MMM$ is said to be pseudo-null if the divisor $\Div(\MMM)$ associated to $\MMM$ in $Z^1(\RRR)$ equals zero. We will let $\MMM_\pn$ denote the maximal pseudo-null submodule of an $\RRR$-module $\MMM$. If $\MMM$ is a pseudo-null $\RRR$-module, we define an element $c_2(\MMM)$ in $Z^2(\RRR)$ as the following formal sum:
\begin{align}
  c_2(\MMM) := \sum_{\substack{\QQQ \subset \RRR \\ \mathrm{height}(\QQQ)=2}} \left(\len_{\RRR_\QQQ} M_\QQQ\right) \QQQ.
\end{align}
In the above formula, the summation is taken over all the height two prime ideals $\QQQ$ of $\RRR$. The invariant $c_2$ in $Z^2(\RRR)$ is a generalization of the ``characteristic divisor'' in $Z^1(\RRR)$ associated to a torsion $\RRR$-module.

Our first main result is Theorem \ref{maintheorem}. We will explain the notations that we have used and some of the motivation behind the hypotheses that we have imposed after the statement of Theorem~\ref{maintheorem}. This theorem follows from Theorem 4.3 of our earlier work \cite{lei2018codimension}. This is explained in Section~\ref{sec:proof_main}.

\begin{Theorem} \label{maintheorem}
Let $\left(\rho,R,d,n\right)$ equal one of the following tuples:
\begin{align*}
\left(\rhoDN{4}{3},R_{F,G}[[\Gamma_\Cyc]],4,3\right), \qquad \left(\rhoDN{8}{4},R_{F,G,H}[[\Gamma_\Cyc]],8,4\right).
\end{align*}
In addition to the hypotheses \ref{hyp:residual-irr} and \ref{hyp:p-distinguished}, we will also suppose that the following conditions hold:
\begin{enumerate}[style=sameline, align=left, label={\textsc{MC}  \textemdash} , ref=\textsc{MC}]
\item\label{hyp:imc} The Iwasawa-Greenberg Main Conjectures for the Galois representation $\rho$ hold, corresponding to both the $p$-adic $L$-functions $\theta^\One_{\pmb{d},n}$ and $\theta^\Two_{\pmb{d},n}$.
\end{enumerate}
\begin{enumerate}[style=sameline, align=left, label={\textsc{GCD} --- }, ref=\textsc{GCD}]
\item\label{hyp:gcd} The height of the ideal $(\theta^{I}_{\pmb{d},n}, \theta^\Two_{\pmb{d},n})$ in the ring $R$ is greater than or equal to two.
\end{enumerate}
\begin{enumerate}[style=sameline,  style=sameline, align=left,label=\textsc{GOR}  --- , ref=\textsc{GOR}]
\item \label{hyp:gor} The  rings $R_F$, $R_G$ and $R_H$ are Gorenstein local rings.
\end{enumerate}
\begin{enumerate}[ style=sameline, align=left, label={$\textsc{Loc}\mathrm{_p(0)}$ --- }, ref={$\textsc{Loc}\mathrm{_p(0)}$}]
\item\label{hyp:locp0} Neither the trivial representation nor the Teichm\"uller character $\omega$ is a Jordan-Holder component of the residual representation $\overline{\rho}$ for the action of $\Gal{\overline{\Q}_p}{\Q_p}$.
\end{enumerate}

Then, we have the following equality in $Z^2\left(R\right)$:
  \begin{align*}
    c_2\left(\frac{R}{\left(\theta^\One_{\pmb{d},n}, \theta^\Two_{\pmb{d},n}\right)}\right) = c_2\bigg(\mathcal{Z}(\Q,D_{\pmb{d},n})\bigg) + c_2\bigg(\mathcal{Z}(\Q,D^\star_{\pmb{d},n})\bigg) + \sum_{l \in \Sigma_0} c_2\bigg(\left(H^0\left(\Q_l, D_{\pmb{d},n}\right)^\vee\right)_{\pn}\bigg).
  \end{align*}
\end{Theorem}

We have imposed the hypotheses \ref{hyp:residual-irr} and \ref{hyp:p-distinguished} to ensure that the $2$-dimensional Galois representations $\rho_F$, $\rho_G$ and $\rho_H$ satisfy Greenberg's Panchishkin condition. This, in turn, lets us deduce that the Galois representation $\rho$ satisfies the Panchishkin condition(s). Following Greenberg's methodology, one can correspondingly construct discrete Selmer groups $\Sel_\One(\Q,D_{\pmb{d},n})$, $\Sel_\Two(\Q,D_{\pmb{d},n})$ and formulate main conjectures relating the $p$-adic $L$-functions $\theta^\One_{\pmb{d},n}$ and $\theta^\Two_{\pmb{d},n}$ to the corresponding Selmer groups. Along with the hypotheses of our main theorem, these main conjectures should predict the following equality of divisors in $Z^1\left(R\right)$:
\begin{align} \label{IMC}
\tag{IMC} \Div\left(\Sel_\One(\Q,D_{\pmb{d},n})^\vee\right) \stackrel{?}{=} \Div(\theta^\One_{\pmb{d},n}), \qquad  \Div\left(\Sel_\Two(\Q,D_{\pmb{d},n})^\vee\right) \stackrel{?}{=} \Div(\theta^\Two_{\pmb{d},n}).
\end{align}
Implicit in the statement of the Iwasawa-Greenberg main conjectures for these Galois representations is the assertion that the Pontryagin duals of these Selmer groups are torsion.  Also implicit is the assertion that the $p$-adic $L$-functions $\theta^\One_{\pmb{d},n}$ and $\theta^\Two_{\pmb{d},n}$ are non-zero and integral (that is, they belong to $R$).

The hypotheses \ref{hyp:gor} and \ref{hyp:locp0} are technical hypotheses that we need to impose to invoke results from our earlier work \cite{lei2018codimension}.

The $R$-module $\mathcal{Z}(\Q,D_{\pmb{d},n})$ is the Pontryagin dual of the  intersection $\Sel_\One(\Q,D_{\pmb{d},n}) \ \bigcap \  \Sel_\Two(\Q,D_{\pmb{d},n})$ inside the first global Galois cohomology group $H^1\left(G_\Sigma,D_{\pmb{d},n}\right)$. One can similarly associate a discrete $R$-module $D^\star_{\pmb{d},n}$ and two discrete Selmer groups $\Sel_\One(\Q,D^\star_{\pmb{d},n})$ and  $\Sel_\Two(\Q,D^\star_{\pmb{d},n})$ to the Galois representation $\rho^\star$ (the Tate dual of $\rho$). {See Section \ref{sec:selmer_struc} for the definition of these objects}. Under the hypotheses of Theorem \ref{maintheorem}, the finitely generated $R$-module $\mathcal{Z}(\Q,D^\star_{\pmb{d},n})$ is the Pontryagin dual of the  intersection $\Sel_\One(\Q,D^\star_{\pmb{d},n}) \ \bigcap \  \Sel_\Two(\Q,D^\star_{\pmb{d},n})$ inside $H^1\left(G_\Sigma,D^\star_{\pmb{d},n}\right)$. See Section \ref{sec:desc} for a precise description of these modules. Under these hypotheses, the $R$-modules $\mathcal{Z}(\Q,D_{\pmb{d},n})$ and $\mathcal{Z}(\Q,D^\star_{\pmb{d},n})$ turn out to be  pseudo-null. See \cite[Proposition 4.1]{lei2018codimension}.

The main impetus behind our results in this paper and prior results in higher codimension Iwasawa theory \cite{bleher2015higher,lei2018codimension}, are pseudo-nullity conjectures in Iwasawa theory. These conjectures are  briefly recalled in Section \ref{sec:pseudonull}. They involve various situations where the Pontryagin dual of the \textit{fine Selmer group} is pseudo-null. We impose the hypothesis labeled \ref{hyp:gcd} since it allows us to work in an instructive setup when the Pontryagin dual of the fine Selmer group is pseudo-null.

We view the terms $c_2\bigg(\left(H^0\left(\Q_l, D_{\pmb{d},n}\right)^\vee\right)_{\pn}\bigg)$ as local fudge factors at  primes  $l \in \Sigma_0$.

\begin{remark}
Buyukboduk and Ochiai \cite{KazimOchiai}, employing the Euler system of Beilinson-Flach elements constructed by Kings-Loeffler-Zerbes \cite{KLZ}, have proved that, under certain technical hypotheses, the inequality $\le$ in \eqref{IMC} holds for the Galois representation $\rhoDN{4}{3}$.
\end{remark}

\begin{remark}
{Theorem 4.3 of our earlier work \cite{lei2018codimension} is a general result. Following the notations in \cite{lei2018codimension}, we associated a pseudo-null module $\mathcal{Z}^{(*)}(\Q,D^\star_{\pmb{d},n})$ to the discrete module  $D^\star_{\pmb{d},n}$ for the Tate dual of the initial Galois representation. Although a priori it may seem that the description for $\mathcal{Z}^{(*)}(\Q,D^\star_{\pmb{d},n})$ need not match the description of $\mathcal{Z}(\Q,D^\star_{\pmb{d},n})$, what we prove in Section \ref{sec:desc} is that in the setting of Theorem \ref{maintheorem}, these descriptions turn out to be equivalent.}
\end{remark}

\subsection{Selmer groups and principal divisors}

Though the overall goal is to associate analytic invariants to the pseudo-null modules $\mathcal{Z}(\Q,D_{\pmb{d},n})$ and $\mathcal{Z}(\Q,D^\star_{\pmb{d},n})$, assuming the validity of the main conjectures allows us to work with the principal divisors $\Div\left(\Sel_\One(\Q,D_{\pmb{d},n})^\vee\right)$ and $\Div\left(\Sel_\Two(\Q,D_{\pmb{d},n})^\vee\right)$ on the algebraic side instead of the corresponding $p$-adic $L$-functions on the analytic side.   An added difficulty that appears in the setting of Theorem \ref{maintheorem} is that the deformation rings $R_F$, $R_G$ and $R_H$  are not always known to be UFDs. {See section 4 in \cite{palvannan2015homological} for some examples when these rings are not UFDs}. Contrast this with earlier results in higher codimension Iwasawa theory \cite{bleher2015higher,lei2018codimension}, which were obtained over regular local rings. Without assuming the Iwasawa-Greenberg main conjecture, the following subtle question arises:

\begin{question} \label{ques:prin}
If the Pontryagin dual of the Selmer group is torsion, does it generate a principal divisor?
\end{question}

We now state our second main result (Theorem \ref{theorem:second}) which provides an affirmative answer to the above question  in the fairly general context of cyclotomic twist deformations\footnote{the definition of cyclotomic twist deformations is recalled in Section \ref{sec:cyc_twist}.}. This confirms a prediction of Ralph Greenberg stated in the paragraph immediately following the statement formulating the Iwasawa main conjecture \cite[Conjecture 4.1]{greenberg1994iwasawa}.\\

Let $\RRR$ be a finite integral extension of $\Z_p[[u_1,\dotsc,u_n]]$ and assume that it is an integrally closed domain. Consider a Galois representation $\varrho : G_\Sigma \rightarrow \Gl_d(\RRR)$ satisfying Greenberg's Panchishkin condition. Let $\varrho \otimes \kappa^{-1}: G_\Sigma \rightarrow \Gl_d(\RRR[[\Gamma_\Cyc]])$ denote the associated cyclotomic twist deformation. One can associate a (primitive) Selmer group $\Sel\left(\Q,\D_{\varrho \otimes \kappa^{-1}}\right)^\vee$ to the Galois representation $\varrho \otimes \kappa^{-1}$. See Section \ref{sec:selmer_struc}. The construction of the Selmer group depends on the choice of the Panchishkin condition.

\begin{Theorem} \label{theorem:second}
Suppose that the $\RRR[[\Gamma_\Cyc]]$-module $\Sel\left(\Q,\D_{\varrho \otimes \kappa^{-1}}\right)^\vee$ is torsion. Then, the following divisor in $Z^1\left(\RRR[[\Gamma_\Cyc]]\right)$ is principal:
\begin{align} \label{eq:divisor_MC}
\Div\left(\Sel\left(\Q,\D_{\varrho \otimes \kappa^{-1}}\right)^\vee\right) - \Div\left(H^0\left(\Q,\D_{\varrho \otimes \kappa^{-1}}\right)^\vee\right) - \Div\left(H^0\left(\Q,\D^\star_{\varrho \otimes \kappa^{-1}}\right)^\vee\right).
\end{align}
\end{Theorem}

In its complete generality, Greenberg's formulation of the Iwasawa main conjecture equates the divisor in equation (\ref{eq:divisor_MC}) to the divisor of the $p$-adic $L$-function. In fact, our results show that the divisors $\Div\left(\Sel\left(\Q,\D_{\varrho \otimes \kappa^{-1}}\right)^\vee\right)$, $\Div\left(H^0\left(\Q,\D_{\varrho \otimes \kappa^{-1}}\right)^\vee\right)$ and $\Div\left(H^0\left(\Q,\D^\star_{\varrho \otimes \kappa^{-1}}\right)^\vee\right)$ are all principal. This last statement relies on the fact that we are working with cyclotomic twist deformations. Our approach towards proving Theorem \ref{theorem:second} involves invoking results on Galois cohomology groups in Greenberg's foundational works  \cite{MR2290593,greenberg2010surjectivity}. Another key input to our approach is the simple observation that for any local ring $\RRR$,  the connecting homomorphism  $\partial: K_1\left(\mathrm{Frac}(\RRR)\right) \rightarrow K_0\left(\RRR,\mathrm{Frac}(\RRR)\right)$ in $K$-theory is surjective. For Galois representations that are not necessarily cyclotomic twist deformations, we prove a non-primitive analog of Theorem \ref{theorem:second} in Proposition \ref{prop:prin_div}. The divisor involving the non-primitive Selmer group appearing in that proposition should also appear in the formulation of a non-primitive version of the Iwasawa main conjecture, relating it to the divisor of a non-primitive $p$-adic $L$-function.

\begin{remark}
In the setting of Theorem \ref{maintheorem}, the hypothesis \ref{hyp:locp0} lets us deduce that both $H^0\left(\Q,D_{\pmb{d},n}\right)$ and $H^0\left(\Q,D^\star_{\pmb{d},n}\right)$ are equal to $0$.
\end{remark}

\begin{remark} \label{rem:torsion_2}
We must emphasize that it is a deep arithmetic question (which we do not address) to establish that the Pontryagin dual of the Selmer group  is torsion.

Using results of Kings, Loeffler and Zerbes in \cite{KLZ}, it would be possible to establish under certain hypotheses that the $R_{F,G}[[\Gamma_\Cyc]]$-module $\Sel_\One\left(\Q,D_{\pmb{4},3}\right)^\vee$ is torsion. One can consider the Selmer groups associated to classical specializations of the Hida families $F$, $G$ with weight $k$, $k'$ respectively and such that $k- k'\ge2$, along with the specialization of the cyclotomic variable to $k-1$. As explained in Remark 11.6.5 and Theorem 11.6.6 of \cite{KLZ}, the Selmer group associated to such a specialization is shown to be finite if we assume that the Galois representation associated to this specialization satisfies a ``big image'' hypothesis and admits ``no exceptional zero'' (these conditions are labeled Hyp(\textit{BI}) and Hyp(\textit{NEZ}) in \cite{KLZ}). When these conditions hold, one can use control theorems to deduce that the $R_{F,G}[[\Gamma_\Cyc]]$-module $\Sel_\One\left(\Q,D_{\pmb{4},3}\right)^\vee$  is torsion. By similar arguments, one can deduce that the $R_{F,G}[[\Gamma_\Cyc]]$-module $\Sel_\Two\left(\Q,D_{\pmb{4},3}\right)^\vee$ is torsion.

To the best of our knowledge, in the setting of the cyclotomic twist deformation of the tensor product of three Hida families, little progress has been made in obtaining general results proving that the Pontryagin dual of the Selmer groups are torsion.
\end{remark}

\begin{remark}
For Theorem \ref{theorem:second}, the reader might also be interested in independent but related results in the thesis of Jyoti Prakash Saha \cite{saha:tel-01124363} and his preprint \cite{saha2018algebraic}.
\end{remark}

\subsection{Pseudo-nullity conjectures in Iwasawa theory} \label{sec:pseudonull}

To briefly motivate the pseudo-nullity conjectures in Iwasawa theory, it will be helpful to consider a general setup. Let $K$ be a number field. Let $\Sigma$ denote a finite set of primes in $K$ containing all the primes lying above $p$ and $\infty$. Let $K_\Sigma$ denote the maximal extension of $K$ ramified outside $\Sigma$. Let $G_\Sigma$ denote $\Gal{K_\Sigma}{K}$. Let $\RRR$ denote a finite integral extension of the power series ring $\Z_p[[x_1,\ldots, x_n]]$. Consider a Galois representation $\varrho:G_\Sigma \rightarrow \Gl_d(\RRR)$. Let $\LLL_\varrho$ denote the free $\RRR$-module on which $G_\Sigma$ acts to let us obtain the Galois representation $\varrho$. Let $\D_\varrho$ denote the discrete $\RRR$-module $\LLL_\varrho \otimes_\RRR \Hom_\cont\left(\RRR,\frac{\Q_p}{\Z_p}\right)$. We will consider the following discrete $\RRR$-module:
\begin{align*}
\Sha^1\left(K,\D_{\varrho}\right) := \ker\left(H^1\left(K_\Sigma/K, \D_\varrho \right) \rightarrow \prod_{\nu \in S} H^1\left(K_\nu, \D_{\varrho} \right) \right).
\end{align*}

The $\RRR$-module $\Sha^1\left(K,\D_\varrho\right)^\vee$ turns out to be finitely generated. One can ask the following question, which involves finer information about the structure of $\Sha^1\left(K,\D_\varrho\right)^\vee$ as an $\RRR$-module:

\begin{question}\label{ques:pn}
Under what hypotheses on $\varrho$, is the $\RRR$-module $\displaystyle \Sha^1\left(K,\D_\varrho\right)^\vee$ pseudo-null?
\end{question}

We recall three situations when $\displaystyle \Sha^1\left(K,D_\varrho\right)^\vee$ is conjectured to be pseudo-null. In all these situations, the Galois representation arises as a cyclotomic twist deformation.\\

Let $\widetilde{K}_\infty$ denote the compositum of all the $\Z_p$-extensions of $K$. Let $\widetilde{\Gamma}_K$ denote $\Gal{\widetilde{K}_\infty}{K}$. The Galois group $\widetilde{\Gamma}_K$ is a finitely generated $\Z_p$-module with rank (say $r$) greater than or equal to $r_2(K)+1$. Assuming Leopoldt's conjecture, this rank equals $r_2(K)+1$. Here, $r_2(K)$ denotes the number of pairs of complex embeddings of the number field $K$. Let $\Sigma$ denote the set of primes in $K$ containing all the primes lying above $p$ and $\infty$. Let $\widetilde{\kappa}:G_\Sigma \twoheadrightarrow \widetilde{\Gamma}_K \rightarrow \Gl_1\left(\Z_p[[\widetilde{\Gamma}_K]\right) $ denote the tautological character taking values in the Iwasawa algebra $\Z_p[[\widetilde{\Gamma}_K]]$, which is isomorphic to the power series ring $\Z_p[[x_1,\ldots, x_r]]$. We have the following pseudo-nullity conjecture, which follows from a conjecture of Greenberg (Conjecture 3.5 in \cite{MR1846466}):

\begin{conjecture} \label{conj:greenberg}
The $\Z_p[[\widetilde{\Gamma}_K]]$-module $\Sha^1\left(K,\D_{\widetilde{\kappa}^{-1}}\right)^\vee$ is pseudo-null.
\end{conjecture}

Now suppose that $K$ is an imaginary quadratic field where the prime $p$ splits. Consider a finite continuous character $\chi:\Gal{\overline{\Q}}{K} \rightarrow \Z_p^\times$. Enlarge the set $\Sigma$ to also contain all the primes ramified in the extension $\overline{\Q}^{\ker(\chi)}/K$. The Iwasawa algebra $\Z_p[[\widetilde{\Gamma}_K]]$ is isomorphic to the power series ring $\Z_p[[x_1,x_2]]$. Consider the Galois representation $\chi\widetilde{\kappa}^{-1}:G_\Sigma \rightarrow \Gl_1\left(\Z_p[[\widetilde{\Gamma}_K]]\right)$.  One of the main results (Theorem 5.2.5 in \cite{bleher2015higher}) in the seven author paper \cite{bleher2015higher} is motivated by the following pseudo-nullity conjecture (which one can view as a modification of Conjecture \ref{conj:greenberg}):

\begin{conjecture}[Conjecture 3.4.1 in \cite{bleher2015higher}] \label{conj:sevenauthor}
The $\Z_p[[\widetilde{\Gamma}_K]]$-module $\Sha^1\left(K,\D_{\chi\widetilde{\kappa}^{-1}}\right)^\vee$ is pseudo-null.
\end{conjecture}

Keeping the same notations as above, let $E$ denote an elliptic curve defined over $\Q$. Let $\Sigma$ denote the finite set of primes in $K$ containing all the primes lying above $p$, $\infty$ and all the primes where the elliptic curve $E$ has bad reduction. Let $T_p(E)$ denote the $p$-adic Tate module. Let $\rho_{E,\widetilde{\kappa}^{-1}}:G_\Sigma \rightarrow \Gl_2(\Z_p[[\widetilde{\Gamma}_K]])$ denote the Galois representation given by the action of $G_\Sigma$ on $T_p(E) \  \hotimes  \ \Z_p[[\widetilde{\Gamma}_K]](\widetilde{\kappa}{^{-1}})$. The main result in an earlier work of ours (Theorem 1 in \cite{lei2018codimension}) is motivated by the following pseudo-nullity conjecture, which is a special case of a conjecture due to Coates and Sujatha (Conjecture B in \cite{MR2148798}):

\begin{conjecture} \label{conj:coates-sujatha}
The $\Z_p[[\widetilde{\Gamma}_K]]$-module $\Sha^1\left(K,\D_{\rho_{E,\widetilde{\kappa}^{-1}}}\right)^\vee$ is pseudo-null.
\end{conjecture}

\begin{remark}
Note that the statement of Conjecture \ref{conj:greenberg} differs from the statement of Conjecture 3.5 in \cite{MR1846466}. The definition of the fine Selmer group $\Sha^1\left(F,\D_{\widetilde{\kappa}^{-1}}\right)$ involves global cocycles in $H^1\left(G_\Sigma,\D_\varrho\right)$ that are ``locally trivial'' at all primes in $S$, whereas in \cite{MR1846466}, Greenberg deals with the subgroup of $H^1\left(K_\Sigma/K,\D_\varrho\right)$ generated by cocycles that are ``unramified'' at all primes in $\Sigma$. Since we are dealing with cyclotomic twist deformations, this distinction is only significant for primes lying above $p$. Nevertheless, Conjecture \ref{conj:greenberg} follows from Conjecture 3.5 in \cite{MR1846466}. Similarly, the statement of Conjecture \ref{conj:sevenauthor} differs from Conjecture 3.4.1 in \cite{bleher2015higher}.
\end{remark}

\begin{remark}
For a general Galois representation $\varrho$, the definition of the fine Selmer group (as defined in Greenberg's work \cite{MR2290593}) would depend on the choice of set $\Sigma$. Since we are only interested in cases when the Galois representation $\varrho$ is a cyclotomic twist deformation, the definition of the fine Selmer group is independent of the set $\Sigma$ as long as the set $\Sigma$ contains the primes lying above $p$, $\infty$ and all the primes of $F$ ramified in the extension $\overline{\Q}^{\ker(\varrho)}/F$.
\end{remark}

Let us follow all the notations used in the statement in the setup of Theorem \ref{maintheorem}. Motivated by these pseudo-nullity conjectures, one can consider the following question:

\begin{question} \label{ques:pn_Hida}
Is the $R$-module $\Sha^1\left(\Q,D_{\pmb{d},n}\right)^\vee$ pseudo-null?
\end{question}

A more speculative question to investigate would be the following:
\begin{question} \label{ques:gcd_Hida}
Is the height of the ideal $(\theta^{I}_{\pmb{d},n}, \theta^\Two_{\pmb{d},n})$ in the ring $R$ greater than or equal to two?
\end{question}

Since we have no numerical evidence towards an affirmative answer for Question \ref{ques:pn_Hida} or Question \ref{ques:gcd_Hida}, we avoid making any conjectures in general. We are especially tentative about considering Question \ref{ques:gcd_Hida}. Nevertheless, it seems instructive to us to work under a setup when Question \ref{ques:gcd_Hida} has an affirmative answer since, assuming the Iwasawa-Greenberg main conjectures, an affirmative answer for Question \ref{ques:gcd_Hida} would imply an affirmative answer for Question \ref{ques:pn_Hida}.

\begin{remark}
Under the setting and hypotheses of Theorem \ref{maintheorem}, we have natural surjections $\mathcal{Z}(\Q,D_{\pmb{d},n}) \twoheadrightarrow \Sha^1\left(\Q,D_{\pmb{d},n}\right)^\vee$ and $\mathcal{Z}(\Q,D^\star_{\pmb{d},n}) \twoheadrightarrow \Sha^1\left(\Q,D^\star_{\pmb{d},n}\right)^\vee$. So, the $R$-modules $\Sha^1\left(\Q,D_{\pmb{d},n}\right)^\vee$ and $\Sha^1\left(\Q,D^\star_{\pmb{d},n}\right)^\vee$ turn out to be pseudo-null.
\end{remark}

\subsection{Concluding remarks}

As we remarked earlier, the validity of the Iwasawa main conjectures will allow us to entirely work on the algebraic side of Iwasawa theory. For this reason, in this paper we will completely avoid specifying the interpolation properties that the $p$-adic $L$-functions need to satisfy at the critical set of specializations. For the $3$-variable Rankin-Selberg $p$-adic $L$-function associated to the cyclotomic twist deformation of the tensor product of two Hida families, we refer to Hida's constructions in his works \cite{hida1985ap,hida1988p,Hidafourier, hida1993elementary}.  As we mentioned earlier, the construction of the $4$-variable $p$-adic $L$-functions associated to the cyclotomic twist deformation of the tensor product of three Hida families is either forthcoming or still conjectural. In light of Question \ref{ques:gcd_Hida}  and the pseudo-nullity conjectures mentioned in Section \ref{sec:pseudonull}, the connection to $p$-adic $L$-functions is still important to us since explicit computations of these $p$-adic $L$-functions may provide a numerical approach to verify the pseudo-nullity conjectures. Though we haven't seriously pursued the idea to numerically verify the pseudo-nullity conjectures in the setup of tensor product of Hida families, it is conceivable that algorithms of Alan Lauder \cite{MR3381453} computing explicit values of $p$-adic Rankin $L$-functions may play an important role.

For the tensor product of three Hida families, Theorem \ref{maintheorem} is not applicable when the $4$-variable $p$-adic $L$-functions are both unbalanced since our methods in \cite{lei2018codimension} do not allow us to handle this case. See Remark \ref{rem:two_unb}. It is also possible for Assumption \ref{hyp:gcd} to fail if we consider a pair of unbalanced $p$-adic $L$-functions.  As an example of a situation when they may have a common height one prime ideal in their support, consider the case when the global root number for the functional equation for the triple product complex $L$-function at the unbalanced weights equals -1 (for instance in the setting of \cite[Theorem B]{hsieh2016hida}), since then the complex $L$-function  would vanish identically at the central value.

To avoid a case-by-case analysis of the reduction types of the Hida families $F$, $G$ and $H$ at primes $l \in \Sigma_0$, we have not explicitly computed the local fudge factors at primes $l \in\Sigma_0$. We refer the reader who is interested in such computations to our earlier work \cite{lei2018codimension}, where we computed similar local fudge factors in the setting of an elliptic curve over $\Q$ with supersingular reduction at $p$ over the compositum of $\Z_p$-extensions of an imaginary quadratic field $K$ where the prime $p$ splits.

\section{Structural properties of Selmer groups} \label{sec:selmer_struc}

Let $\RRR$ be a finite integral extension of $\Z_p[[u_1,\dotsc,u_n]]$ and assume that it is an integrally closed domain. Consider a Galois representation $\varrho : G_\Sigma \rightarrow \Gl_d(\RRR)$, whose associated Galois lattice $\LLL_\varrho$ is free over $\RRR$. Let $d^\pm$ denote the dimension of the $\pm$1-eigenspace for the action of complex conjugation. Suppose also that we can associate a $\Gal{\overline{\Q}_p}{\Q_p}$-equivariant short exact sequence of free $\RRR$-modules:
\begin{align} \label{eq:fil-rho}
\tag{Fil-$\varrho$} 0\rightarrow \Fil\LLL_\varrho \rightarrow \LLL_\varrho \rightarrow \frac{\LLL_\varrho}{\Fil\LLL_\varrho} \rightarrow 0,
\end{align}
Suppose also that the following condition holds:
\begin{align} \label{p-critical}
\tag{$p$-critical}
\mathrm{Rank}_{\RRR}\left(\Fil \LLL_{\varrho}\right) = d^+.
\end{align}

Let $\varrho^\star:G_\Sigma \rightarrow \Gl_d(\RRR)$ denote the Galois representation given by the action of $G_\Sigma$ on the free $\RRR$-module $\LLL^\star_{\varrho}:=\Hom_{\RRR}\left(\LLL_{\varrho}, \RRR(\chi_p)\right)$. Here, $\chi_p:G_\Sigma \twoheadrightarrow \Gal{\Q(\mu_{p^\infty})}{\Q} \xrightarrow \cong  \Z_p^\times$ is the $p$-adic cyclotomic character given by the action of $G_\Sigma$ on the $p$-power roots of unity $\mu_{p^\infty}$. The Galois representation $\varrho^\star$ is referred to as the Tate dual of $\varrho$. We have the following short exact sequence of free $\RRR$-modules that is $\Gal{\overline{\Q}_p}{\Q_p}$-equivariant:
\begin{align*}
0\rightarrow \Fil \LLL^\star_{\varrho} \rightarrow \LLL^\star_{\varrho} \rightarrow \frac{\LLL^\star_{\varrho}}{\Fil \LLL^\star_{\varrho}} \rightarrow 0, \quad \text{where } \Fil \LLL^\star_{\varrho} := \Hom_\RRR\left(\frac{\LLL_{\varrho}}{\Fil \LLL_{\varrho}}, \RRR(\chi_p)\right).
\end{align*}

Thus, we naturally obtain a filtration on $\varrho^\star$  that also satisfies (\ref{p-critical}). We will consider the following discrete $\RRR$-modules:
\begin{align*}
\DDD_\varrho&:= \LLL_{\varrho} \otimes_\RRR \Hom_\cont\left(\RRR, \frac{\Q_p}{\Z_p}\right), \qquad && \Fil \DDD_\varrho:= \Fil \LLL_{\varrho} \otimes_\RRR \Hom_\cont\left(\RRR, \frac{\Q_p}{\Z_p}\right),  \\
\DDD^\star_\varrho&:= \LLL^\star_{\varrho} \otimes_\RRR \Hom_\cont\left(\RRR, \frac{\Q_p}{\Z_p}\right), \qquad && \Fil \DDD^\star_\varrho:= \Fil \LLL^\star_{\varrho} \otimes_\RRR \Hom_\cont\left(\RRR, \frac{\Q_p}{\Z_p}\right).
\end{align*}

\begin{remark}
Following \cite{greenberg1994iwasawa}, one can formulate the Iwasawa-Greenberg main conjecture for the Galois representation $\varrho$ when it satisfies the ``Panchishkin condition''. The Panchishkin condition involves a filtration, such as the one given in equation (\ref{eq:fil-rho}), satisfying (\ref{p-critical}) along with an ``ample'' critical set of ``motivic'' specializations $C \subset \Hom_{\cont}\left(\RRR,\overline{\Q}_p\right)$. We refer the interested reader to Section 4 of \cite{greenberg1994iwasawa} for a precise description of the Panchishkin condition.
\end{remark}

\subsection{Non-primitive Selmer groups and principal divisors}

We shall define a discrete non-primitive Selmer group $\Sel^{\Sigma_0}\left(\Q,\D_{\varrho}\right)$ and a discrete ``strict'' non-primitive Selmer group $\Sel^{\Sigma_0, \str}\left(\Q,\D_{\varrho}\right)$ as follows:
\begin{align*}
  &\Sel^{\Sigma_0}\left(\Q,\D_{\varrho}\right):=\ker\bigg(H^1\left(G_\Sigma, \D_{\varrho}\right)  \xrightarrow {\phi^{\Sigma_0}}  \frac{H^1\left(\Q_p,\D_{\varrho}\right)}{\Loc\left(\Q_p,\D_{\varrho}\right)}  \bigg), \\
   & \Sel^{\Sigma_0,\str}\left(\Q,\D_{\varrho}\right):=\ker\bigg(H^1\left(G_\Sigma, \D_{\varrho}\right)  \xrightarrow {\phi_\str^{\Sigma_0}}  \frac{H^1\left(\Q_p,\D_{\varrho}\right)}{\Loc_{\mathrm{str}}\left(\Q_p,\D_{\varrho}\right)}  \bigg).
\end{align*}
Here, the local conditions at $p$ are given below ($I_p$ denote the inertia subgroup at $p$):
\begin{align*}
&\Loc(\Q_p,\D_{\varrho}):=\ker\left(H^1\left(\Q_p,\D_{\varrho}\right) \rightarrow H^1\left(I_p,\frac{\D_{\varrho}}{\Fil \D_{\varrho}}\right) \right), \\ &\Loc_{\mathrm{str}}(\Q_p,\D_{\varrho}):=\ker\left(H^1\left(\Q_p,\D_{\varrho}\right) \rightarrow H^1\left(\Q_p,\frac{\D_{\varrho}}{\Fil \D_{\varrho}}\right) \right).
\end{align*}

\begin{proposition} \label{prop:prin_div}
Suppose that the $\RRR$-module $\Sel^{\Sigma_0}\left(\Q,\D_{\varrho}\right)^\vee$ is torsion. In addition, suppose that the following conditions hold:
\begin{enumerate}
\item\label{cond:one} The $\RRR$-modules $H^0\left(\Q_\Sigma/\Q, \D_\varrho\right)^\vee$ and $H^0\left(I_p,\frac{\D_\varrho}{\Fil \D_\varrho}\right)^\vee$ are torsion.
\item\label{cond:two} $H^2\left(\Q_\Sigma/\Q,\D_\varrho\right)= H^2\left( \Q_p,\Fil \D_\varrho\right) = H^2\left( \Q_p,\frac{\D_\varrho}{\Fil \D_\varrho}\right)=0$.
\item\label{cond:three} There exists a finite prime $l \neq p$ such that $H^2\left(\Q_{l},\D_\varrho\right)=0$.
\end{enumerate}
Then, the following divisor  in $Z^1\left(\RRR\right)$ is principal:
\begin{align*}
\Div\left(\Sel^{\Sigma_0}(\Q,\D_\varrho)^\vee\right)  - \Div\left(H^0\left(\Q_\Sigma/\Q, \D_\varrho\right)^\vee\right).
\end{align*}
\end{proposition}

Before  proving Proposition \ref{prop:prin_div}, it will be helpful to prove the following lemma. Let $\SSS$ denote the multiplicatively closed set $\RRR \setminus \{0\}$.  Let  $\mathbf{Ch}_{\SSS}^\flat\big(\mathbf{P}(\RRR)\big)$ denote the Waldhausen category of bounded chain complexes of finitely generated projective $\RRR$-modules whose cohomologies are $\SSS$-torsion.

\begin{lemma} \label{lem:K0eul_ch}
Suppose $\MMM$ and $\NNN$ are two finitely generated torsion $\RRR$-modules such that $[\MMM] -[\NNN]$ is the Euler characteristic of an element of $K_0\left(\mathbf{Ch}_{\SSS}^\flat\big(\mathbf{P}(\RRR)\big)\right)$. Then, the divisor $\Div(\MMM) - \Div(\NNN)$ in $Z^1(\RRR)$ is principal.
\end{lemma}

\begin{proof}
Let $Q_\RRR$ denote the fraction field of $\RRR$. Note that $Q_\RRR$  is obtained as the localization of $\RRR$ at the multiplicatively closed set $\SSS$. By Theorem 3.2 in Weibel's K-book \cite{MR3076731}, we have the following commutative diagram, where the top row of $K$-groups is exact:

\begin{align*}
\xymatrix{
K_1\left(\RRR\right) \ar[r]\ar[d]^{\det}_{\cong} & K_1\left(Q_\RRR \right) \ar[r]^{\partial \qquad} \ar[d]^{\det}_{\cong}  & K_0\left(\mathbf{Ch}_{\SSS}^\flat\big(\mathbf{P}(\RRR)\big)\right) \ar[r]& K_0\left(\RRR\right) \ar[r]\ar[d]^{\mathrm{Rank}}_{\cong}& K_0\left(Q_\RRR\right) \ar[d]^{\dim}_{\cong} \ar[r]& 0 \\
\RRR ^\times \ar[r]& Q_\RRR^\times & & \Z \ar[r]& \Z
}
\end{align*}

Note that $K_0\left(\RRR\right)$ is isomorphic to $\Z$ by sending the class of a finitely generated projective module $\RRR^n$ over $\RRR$  to its rank $n$ (every finitely generated projective module over a Noetherian local ring is free). See Lemma 2.2 in Weibel's K-book \cite{MR3076731}. Similarly, $K_0\left(Q_\RRR\right)$ is isomorphic to $\Z$ by sending the class of a finite-dimensional vector space $Q_\RRR^n$ over $Q_\RRR$  to its dimension $n$. The natural morphism $K_0\left(\RRR\right) \rightarrow K_0\left(Q_\RRR\right)$ involves sending the class of $\RRR^n$ to the class of $Q_\RRR^n$. Consequently, the map $K_0\left(\RRR\right) \rightarrow K_0\left(Q_\RRR\right)$ is an isomorphism. As a result, the connecting morphism $\partial$ is surjective. Let $\theta$ be an element of $Q_\RRR$ such that $\partial(\theta)$ equals $[\MMM] -[\NNN]$ in $K_0\left(\mathbf{Ch}_{\SSS}^\flat\big(\mathbf{P}(\RRR)\big)\right)$.

Let $\p$ denote a height one prime ideal in $\RRR$. Let $\SSS_\p$ denote the multiplicatively closed set $\RRR \setminus \p$. By an argument similar to the one given in the previous paragraph, we have the following commutative diagram whose rows are exact:
\begin{align*}
\xymatrix{
& K_1\left(\RRR_\p\right) \ar[r]\ar[d]^{\det}_{\cong} & K_1\left(Q_\RRR \right) \ar[r]^{\partial_\p \qquad} \ar[d]^{\det}_{\cong}  & K_0\left(\mathbf{Ch}_{\SSS_\p}^\flat\big(\mathbf{P}(\RRR_\p)\big)\right) \ar[d]^{\cong} \ar[r]& 0 \\
0 \ar[r]& \RRR_p^\times \ar[r]& Q_\RRR^\times \ar[r]^{\mathrm{val_\p}}& \Z \ar[r]& 0.
}
\end{align*}
Note that since $\p$ is a height one prime in an integrally closed domain, the ring $\RRR_p$ is a DVR. We have a natural isomorphism $\frac{Q_\RRR^\times}{\RRR_\p^\times} \cong \Z$ given by the $\p$-adic valuation (denoted $\mathrm{val}_\p$).  Consider free resolutions $0 \rightarrow  \RRR_\p^m \xrightarrow {A_M} \RRR_\p^m \rightarrow \MMM_\p$ and $0 \rightarrow  \RRR_\p^n \xrightarrow {A_N} \RRR_\p^n \rightarrow \NNN_\p$  of the finitely generated torsion $\RRR_\p$-modules $\MMM_\p$ and $\NNN_\p$ respectively. Set $\mathrm{val}_\p \left(\MMM_\p\right)$ and $\mathrm{val}_\p \left(\NNN_\p\right)$ to equal $\mathrm{val}_\p\left(\det(A_M)\right)$ and $\mathrm{val}_\p\left(\det(A_N)\right)$ respectively. The isomorphism $K_0\left(\mathbf{Ch}_{\SSS_\p}^\flat\big(\mathbf{P}(\RRR_\p)\big)\right) \cong \Z$ sends $[\MMM_\p]-[\NNN_\p]$ to $\mathrm{val}_\p \left(\MMM_\p\right) - \mathrm{val}_\p \left(\NNN_\p\right)$. See Corollary 3.1.1 in Weibel's K-book \cite{MR3076731}. Note that $\partial_\p(\theta)$ equals $[\MMM_\p]-[\NNN_\p]$ in $K_0\left(\mathbf{Ch}_{\SSS_\p}^\flat\big(\mathbf{P}(\RRR_\p)\big)\right)$. As a result, $\mathrm{val}_\p (\theta)$  equals $\mathrm{val}_\p \left(\MMM_\p\right) - \mathrm{val}_\p \left(\NNN_\p\right)$. So, we have the following equality of divisors in $Z^1\left(\RRR\right)$:
\begin{align*}
\Div\left(\theta\right) = \Div(\MMM) - \Div(\NNN).
\end{align*}
The lemma follows.
\end{proof}

\begin{proof}[Proof of Proposition \ref{prop:prin_div}]
Let $\Gamma_p$ denote the quotient $\frac{\Gal{\overline{\Q}_p}{\Q_p}}{I_p}$. Let $\Frob_p$ denote the Frobenius at $p$, which is a topological generator for $\Gamma_p$. Consider the following exact sequence of  $\RRR$-modules:
\begin{align*}
0 \rightarrow H^0\left(\Q_p,\frac{\D_\varrho}{\Fil \D_\varrho}\right) \rightarrow H^0\left(I_p,\frac{\D_\varrho}{\Fil \D_\varrho}\right) \xrightarrow {\Frob_p-1} H^0\left(I_p,\frac{\D_\varrho}{\Fil \D_\varrho}\right) \rightarrow H^1\left(\Gamma_p,H^0\left(I_p,\frac{\D_\varrho}{\Fil \D_\varrho}\right)\right) \rightarrow 0.
\end{align*}
Condition (\ref{cond:one}) stated in the proposition tells us that the Pontryagin dual of every term in the above exact sequence is torsion. So, we obtain the following equality of divisors in $Z^1\left(\RRR\right)$:

\begin{align} \label{eq:div_h0}
\Div\left(H^0\left(\Q_p,\frac{\D_\varrho}{\Fil \D_\varrho}\right)^\vee \right) = \Div\left(H^1\left(\Gamma_p,H^0\left(I_p,\frac{\D_\varrho}{\Fil \D_\varrho}\right)\right)^\vee\right).
\end{align}

Consider the following commutative diagram where the bottom row is exact:
\begin{align*}
\xymatrix{ & & H^1\left(\Q_\Sigma/\Q,\D_\varrho\right) \ar[d]^{\phi^{\Sigma_0}_\str} \ar[r]^{\cong}& H^1\left(\Q_\Sigma/\Q,\D_\varrho\right) \ar[d] \\ 0 \ar[r]& H^1\left(\Gamma_p,H^0\left(I_p,\frac{\D_\varrho}{\Fil \D_\varrho}\right)\right) \ar[r]& H^1\left(\Q_p,\frac{\D_{\varrho}}{\Fil \D_{\varrho}}\right)   \ar[r]& H^1\left(I_p,\frac{\D_{\varrho}}{\Fil \D_{\varrho}}\right)^{\Gamma_p} \ar[r]& 0
}
\end{align*}

Since $H^2\left(\Q_p,\Fil \D_\varrho\right)$ equals zero,  $\frac{H^1\left(\Q_p,\D_{\varrho}\right)}{\Loc^{\mathrm{str}}\left(\Q_p,\D_{\varrho}\right)}$ is isomorphic as an $\RRR$-module to $H^1\left(\Q_p,\frac{\D_{\varrho}}{\Fil D_{\varrho}}\right)$. Conditions (\ref{cond:one}) and (\ref{cond:two}) along with the fact that the $\RRR$-module $\Sel^{\Sigma_0}\left(\Q,D_{\varrho}\right)^\vee$ is torsion let us deduce that the Weak Leopoldt conjecture for $\varrho$ holds and that the Pontryagin dual of $\coker(\phi^{\Sigma_0}_\str)$ is torsion. See the illustration involving global and local Euler-Poincair\'e characteristics after Lemma 4.5 in \cite{palvannan2016algebraic}.  Proposition 3.2.1 in Greenberg's work \cite{greenberg2010surjectivity} along with Condition (\ref{cond:three}) now let us conclude that the global-to-local map $\phi^{\Sigma_0}_\str$ defining the non-primitive strict Selmer group is surjective. The Snake Lemma then lets us obtain the following exact sequence of $\RRR$-modules:
\begin{align} \label{ses:str_imp}
0 \rightarrow \Sel^{\Sigma_0,\str}\left(\Q,\D_{\varrho}\right) \rightarrow  \Sel^{\Sigma_0}\left(\Q,\D_{\varrho}\right) \rightarrow H^1\left(\Gamma_p,H^0\left(I_p,\frac{\D_\varrho}{\Fil \D_\varrho}\right)\right) \rightarrow  0.
\end{align}

Combining equations (\ref{eq:div_h0}) and (\ref{ses:str_imp}) lets us obtain the following equality of divisors in $Z^1\left(\RRR\right)$:
\begin{align}\label{eq:str-imp}
\Div\left(\Sel^{\Sigma_0}\left(\Q,\D_{\varrho}\right)^\vee\right) = \Div\left(\Sel^{\Sigma_0,\str}\left(\Q,\D_{\varrho}\right)^\vee\right) + \Div\left(H^0\left(\Q_p,\frac{\D_\varrho}{\Fil \D_\varrho}\right)^\vee \right).
\end{align}

Now consider the completed group rings $\RRR[[\Gal{\Q_\Sigma}{\Q}]]$ and $\RRR[[\Gal{\overline{\Q}_p}{\Q_p}]]$ (as in Chapter V, \S 2 in \cite{neukirch2008cohomology}) and the corresponding augmentation ideals denoted $\mathcal{I}_{\Q_p} $ and $\mathcal{I}_{\Q_\Sigma}$ respectively. We have the following exact sequences:
\begin{align} \label{eq:tensor}
0 \rightarrow \mathcal{I}_{\Q_p} \rightarrow \RRR[[\Gal{\overline{\Q}_p}{\Q_p}]] \rightarrow \RRR \rightarrow 0, \qquad 0 \rightarrow \mathcal{I}_{\Q_\Sigma} \rightarrow \RRR[[\Gal{\Q_\Sigma}{\Q}]] \rightarrow \RRR \rightarrow 0.
\end{align}

It is possible to describe the global and local Galois cohomology groups using $\Tor$ groups. See Proposition 5.2.7 and Corollary 5.2.9 in \cite{neukirch2008cohomology}.  For each $i \geq 0$, we have the following isomorphisms of $\RRR$-modules:
\begin{align} \label{eq:tor_coh}
\Tor^{\RRR[[\Gal{\overline{\Q}_p}{\Q_p}]]}_{i}\left(\RRR,\ \left(\frac{\D_\varrho}{\Fil \D_\varrho}\right)^\vee\right) \cong H^i\left(\Q_p,\frac{\D_\varrho}{\Fil \D_\varrho}\right)^\vee,  \ \Tor^{\RRR[[\Gal{{\Q}_\Sigma}{\Q}]]}_{i}\left(\RRR, \ \D_\varrho^\vee\right) \cong H^i\left(\Q_\Sigma/\Q,\D_\varrho\right)^\vee.
\end{align}

By considering the tensor products $(-) \otimes_{\RRR[[\Gal{\overline{\Q}_p}{\Q_p}]]} \RRR$  and $(-) \otimes_{\RRR[[\Gal{\Q_\Sigma}{\Q}]]} \RRR$ respectively over the short exact sequences in equation (\ref{eq:tensor}), we have the following natural commutative diagram whose rows are exact:
\begin{align*}
\xymatrix{
 0 \ar[r]& H^1\left( \Q_p,\frac{\D_\varrho}{\Fil \D_\varrho}\right)^\vee \ar[d]^{\alpha_1} \ar[r]& \mathcal{I}_{\Q_p} \otimes_{\RRR[[\Gal{\overline{\Q}_p}{\Q_p}]]} \left(\frac{\D_\varrho}{\Fil \D_\varrho}\right)^\vee \ar[d]^{\alpha_2} \ar[r]& \left(\frac{\D_\varrho}{\Fil \D_\varrho}\right)^\vee \ar[d]^{\alpha_3} \ar[r]& H^0\left( \Q_p,\frac{\D_\varrho}{\Fil \D_\varrho}\right)^\vee \ar[d]^{\alpha_4} \ar[r]& 0  \\
 0 \ar[r]& H^1\left( \Q_\Sigma/\Q,\D_\varrho\right)^\vee \ar[r]& \mathcal{I}_{\Q_\Sigma} \otimes_{\RRR[[\Gal{\Q_\Sigma}{\Q}]]} \D_\varrho^\vee \ar[r]& \D_\varrho^\vee \ar[r]& H^0\left( \Q_\Sigma/\Q,\D_\varrho\right)^\vee \ar[r]& 0.
}
\end{align*}

As we discussed earlier, the global-to-local map $\phi^{\Sigma_0}_\str$ defining the non-primitive strict Selmer group is surjective. By considering Pontryagin duals, one can conclude that the map $\alpha_1$ is injective. The natural surjection $\D_\varrho \twoheadrightarrow \frac{\D_\varrho}{\Fil \D_\varrho}$ provides us an injection $\left(\frac{\D_\varrho}{\Fil \D_\varrho} \right)^\vee \hookrightarrow \D_\varrho^\vee$. The five lemma then lets us conclude that the map $\alpha_2$ is also an injection. Studying the kernel-cokernel\footnote{Though the Snake Lemma is not directly applicable since the rows of the commutative diagram involve four terms, the fact that the maps $\alpha_1$, $\alpha_2$ and $\alpha_3$ are injective allows us to deduce equation (\ref{eq:longexfour}). One can split the above commutative diagram into two commutative diagrams with two rows each (with each row being a short exact sequence) and then apply the snake lemma to each commutative diagram.} exact sequence involving the above commutative diagram, we obtain the following exact sequence of $\RRR$-modules:

\begin{align} \label{eq:longexfour}
0 \rightarrow \ker(\alpha_4) \rightarrow \frac{\coker(\alpha_2)}{\Sel^{\Sigma_0,\str}\left(\Q,\D_{\varrho}\right)^\vee}  \rightarrow \coker(\alpha_3) \rightarrow \coker(\alpha_4) \rightarrow 0.
\end{align}

Condition (\ref{cond:three}) lets us conclude that both the domain and codomain of the map $\alpha_4$ are torsion $\RRR$-modules. This lets us obtain the following equality of divisors in $Z^1\left(\RRR\right)$:
\begin{align}\label{eq:tordivisors}
\Div\left(\ker(\alpha_4)\right) -  \Div\left(\coker(\alpha_4)\right) = \Div\left(H^0\left(\Q_p,\frac{\D_\varrho}{\Fil \D_\varrho}\right)^\vee \right) -  \Div\left(H^0\left( \Q_\Sigma/\Q,\D_\varrho\right)^\vee\right).
\end{align}

The $p$-cohomological dimensions of the Galois groups $\Gal{\overline{\Q}_p}{\Q_p}$ and $\Gal{\Q_\Sigma}{\Q}$ are both less than or equal to two. See Theorem 7.1.8 and Proposition 8.3.18 in \cite{neukirch2008cohomology}. By Corollary 5.2.13 in \cite{neukirch2008cohomology}, the projective dimensions (both as a left and as a right module) of $\RRR$ as an $\RRR[[\Gal{\overline{\Q}_p}{\Q_p}]]$-module and as an $\RRR[[\Gal{\Q_\Sigma}{\Q}]]$-module are both less than or equal to two. So, the projective dimensions (both as a left and as a right module) of $\mathcal{I}_{\Q_p}$ as an $\RRR[[\Gal{\overline{\Q}_p}{\Q_p}]]$-module and $\mathcal{I}_{\Q_\Sigma}$ as an $\RRR[[\Gal{\Q_\Sigma}{\Q}]]$-module are both less than or equal to one. Consider  projective resolutions of length one for the $\RRR[[\Gal{\overline{\Q}_p}{\Q_p}]]$-module $\mathcal{I}_{\Q_p}$ and the $\RRR[[\Gal{\Q_\Sigma}{\Q}]]$-module $\mathcal{I}_{\Q_\Sigma}$. By considering the tensor products $\otimes_{\RRR[[\Gal{\overline{\Q}_p}{\Q_p}]]} \left(\frac{\D_\varrho}{\Fil \D_\varrho}\right)^\vee$ and $ \otimes_{\RRR[[\Gal{\Q_\Sigma}{\Q}]]} \D_\varrho^\vee$ respectively, one can conclude that  $\mathcal{I}_{\Q_p} \otimes_{\RRR[[\Gal{\overline{\Q}_p}{\Q_p}]]} \left(\frac{\D_\varrho}{\Fil \D_\varrho}\right)^\vee$ and $\mathcal{I}_{\Q_\Sigma} \otimes_{\RRR[[\Gal{\Q_\Sigma}{\Q}]]} \D_\varrho^\vee$ have projective dimension less than or equal to one as $\RRR$-modules since by condition (\ref{cond:three}) and equations (\ref{eq:tensor}) and (\ref{eq:tor_coh}), we have
\begin{align*}
\Tor^{\RRR[[\Gal{\overline{\Q}_p}{\Q_p}]]}_{1}\left(\mathcal{I}_{\Q_p} ,\ \left(\frac{\D_\varrho}{\Fil \D_\varrho}\right)^\vee\right) & \cong \Tor^{\RRR[[\Gal{\overline{\Q}_p}{\Q_p}]]}_{2}\left(\RRR,\ \left(\frac{\D_\varrho}{\Fil \D_\varrho}\right)^\vee\right) \cong H^2\left(\Q_p,\frac{\D_\varrho}{\Fil \D_\varrho}\right)^\vee = 0,  \\ \Tor^{\RRR[[\Gal{{\Q}_S}{\Q}]]}_{1}\left(\mathcal{I}_{\Q_\Sigma} , \ \D_\varrho^\vee\right)  &\cong \Tor^{\RRR[[\Gal{{\Q}_S}{\Q}]]}_{2}\left(\RRR, \ \D_\varrho^\vee\right) \cong H^2\left(\Q_\Sigma/\Q,\D_\varrho\right)^\vee = 0.
\end{align*}
Since the map $\alpha_2$ is injective, we obtain that the projective dimension of $\coker(\alpha_2)$ as an $\RRR$-module is finite. Consider a bounded projective resolution $ \mathbf{P}_{\bullet} \rightarrow \coker(\alpha_2)\rightarrow 0$ for the $\RRR$-module $\coker(\alpha_2)$. Note that $\coker(\alpha_3)$ is isomorphic to the free $\RRR$-module $(\Fil \D_\varrho)^\vee$. Combining this observation along with the observations in this paragraph and equation (\ref{eq:longexfour}) lets us obtain  a bounded chain complex $\mathbf{P}_{\bullet} \rightarrow  \left(\Fil \D_\varrho\right)^\vee$ in $\mathbf{Ch}_{\SSS}^\flat\big(\mathbf{P}(\RRR)\big)$ whose Euler characteristic equals $\left[\Sel^{\Sigma_0,\str}\left(\Q,\D_{\varrho}\right)^\vee\right] + \left[\ker(\alpha_4)\right] - \left[\coker(\alpha_4)\right]$. By Lemma \ref{lem:K0eul_ch}, the following divisor in $Z^1\left(\RRR\right)$ is principal:
\begin{align*}
\Div\left(\Sel^{\Sigma_0,\str}\left(\Q,\D_{\varrho}\right)^\vee\right) + \Div\left(\ker(\alpha_4)\right) - \Div\left(\coker(\alpha_4)\right).
\end{align*}
Combining equations (\ref{eq:str-imp}) and (\ref{eq:tordivisors}), we have the following equality of divisors in $Z^1\left(\RRR\right)$:
\begin{align*}
& \Div\left(\Sel^{\Sigma_0,\str}\left(\Q,\D_{\varrho}\right)^\vee \right) + \Div\left(\ker(\alpha_4)\right)- \Div\left(\coker(\alpha_4)\right) \\
=\ &\Div\left(\Sel^{\Sigma_0}\left(\Q,\D_{\varrho}\right)^\vee\right)  - \Div\left(H^0\left(\Q_p,\frac{\D_\varrho}{\Fil \D_\varrho}\right)^\vee \right) +  \Div\left(H^0\left(\Q_p,\frac{\D_\varrho}{\Fil \D_\varrho}\right)^\vee \right) -  \Div\left(H^0\left( \Q_\Sigma/\Q,\D_\varrho\right)^\vee\right) \\= \ &\Div\left(\Sel^{\Sigma_0}\left(\Q,\D_{\varrho}\right)^\vee \right) - \Div\left(H^0\left( \Q_\Sigma/\Q,\D_\varrho\right)^\vee\right).
\end{align*}
Combining these observations completes the proof of the Proposition.
\end{proof}

\begin{lemma} \label{lem:tor_h0dual}
Suppose that the $\RRR$-module $H^0\left(\Q_\Sigma/\Q , \D^\star_{\varrho}\right)^\vee$ is torsion. Then, we have the following equality of divisors in $Z^1\left(\RRR\right)$:
\begin{align*}
\Div\left(H^0\left(\Q_\Sigma/\Q, \D^\star_{\varrho}\right)^\vee\right)  = \Div\left(H^1\left(\Q_\Sigma/\Q, \LLL^\star_{\varrho}\right)_{\RRR-\mathrm{tor}}\right).
\end{align*}
\end{lemma}

\begin{proof}
Observe that the image of the Galois representation $\mathrm{Im}(\varrho^\star)$ inside $\Gl_d(\RRR)$ is topologically finitely generated. To see this, note that we have a short exact sequence $0 \rightarrow \mathfrak{P} \rightarrow \mathrm{Im}(\varrho^\star) \rightarrow \mathrm{Im}(\overline{\varrho^\star})\rightarrow 0$. Here, $\mathrm{Im}(\overline{\varrho^\star})$ denotes the image of the residual representation associated to $\varrho^\star$ and is a finite group. Here, $\mathfrak{P}$ is a pro-$p$ group which, by an application of the Hermite-Minkowski theorem, can be shown to be topologically finitely generated. Let $\varrho^\star(\gamma_1) ,\cdots, \varrho^\star(\gamma_r)$ denote a finite set of topological generators for $\mathrm{Im}(\varrho^\star)$. Let $A$ denote the $d \times dr$ matrix $[\varrho^\star(\gamma_1) -I_d, \cdots \varrho^\star(\gamma_r) - I_d]$. Here $I_d$ is the identity matrix.

Let $\p$ denote a height one prime ideal in $\RRR$. Let $\pi_\p$ denote a uniformizer for the DVR $\RRR_\p$. Let $\MMM$ denote $\left(\D^\star_{\varrho}\right)^\vee$ --- a free $\RRR$-module of rank $d$. The discrete $\RRR$-module $H^0\left(\Q_\Sigma/\Q, \D^\star_{\varrho}\right)$ is isomorphic to $\ker\left(\D^\star_{\varrho} \xrightarrow {A} (\D^\star_{\varrho})^{{r}}\right)$.  By considering Pontryagin duals and localizing at $\p$, we have the following isomorphism of $\RRR_\p$-modules:
\begin{align} \label{eq:disc_inv}
 H^0\left(\Q_\Sigma/\Q , \D^\star_{\varrho}\right)^\vee \otimes_\RRR \RRR_\p & \cong \coker\left(\MMM_\p^{{r}} \xrightarrow {A^T} \MMM_\p\right) \cong \bigoplus \limits_i \frac{\RRR_\p}{(\pi_\p^{n_i})}.
\end{align}

Here, the elements $\pi_\p^{n_i}$ denote the invariant factors of the matrix $A^T$ over the DVR $\RRR_\p$.

By the arguments given in by \cite[Proposition 2.2.2]{greenberg2010surjectivity}, we have the following isomorphism of $\RRR_\p$-modules for large enough $m$:
\begin{align} \label{eq:com_h0}
H^1\left(\Q_\Sigma/\Q, \LLL^\star_{\varrho}\right)_{\RRR-\mathrm{tor}} \otimes_{\RRR} \RRR_\p  \cong H^0\left(\Q_\Sigma/\Q , \frac{\LLL^\star_{\varrho}}{\pi_\p^m}\right) \otimes_{\RRR} \RRR_\p.
\end{align}
Let us choose such an $m$  ensuring also that it is larger than all the integers $m_i$ that appear in the exponent of the invariant factors $\pi_\p^{m_i}$ of the matrix $A$ over the DVR $\RRR_\p$. Note that the $\RRR$-module $H^0\left(\Q_\Sigma/\Q , \frac{\LLL^\star_{\varrho}}{\pi_\p^m}\right)$ is isomorphic to  $\ker\left(\dfrac{\LLL^\star_{\varrho}}{\pi_\p^m \LLL^\star_{\varrho}}  \xrightarrow {A} \left(\dfrac{\LLL^\star_{\varrho}}{\pi_\p^m \LLL^\star_{\varrho}}\right)^{{r}}\right)$. Since the $\RRR$-module $H^0\left(\Q_\Sigma/\Q , \D^\star_{\varrho}\right)^\vee$ is torsion, by \cite[Proposition 3.10]{MR2290593}, the $\RRR$-module $H^0\left(\Q_\Sigma/\Q , \LLL^\star_{\varrho}\right)^\vee$ equals zero.  Consider the short exact sequence of $\RRR$-modules: $0 \rightarrow \LLL^\star_{\varrho} \xrightarrow {A} (\LLL^\star_{\varrho})^{{r}}\rightarrow Y \rightarrow 0$. Consider the tensor product $\otimes_\RRR \frac{\RRR_\p}{(\pi_\p^m)}$, we get the following isomorphism of $\RRR_\p$-modules:
\begin{align} \label{eq:tor_inv}
\notag H^0\left(\Q_\Sigma/\Q , \frac{\LLL^\star_{\varrho}}{\pi_\p^m \LLL^\star_{\varrho}}\right) \otimes_\RRR \RRR_\p & \cong \ker\left(\dfrac{\LLL^\star_{\varrho}}{\pi_\p^m\LLL^\star_{\varrho}}  \xrightarrow {A} \left(\dfrac{\LLL^\star_{\varrho}}{\pi_\p^m \LLL^\star_{\varrho}}\right)^{{r}}\right) \otimes_\RRR \RRR_\p \cong \Tor_1^{\RRR_p}\left(Y_\p, \frac{\RRR_\p}{(\pi_\p^{m})}\right) \cong Y_\p [\pi_\p^m] \\  & \cong \bigoplus_i  \dfrac{\RRR_p}{(\pi^{m_i}_\p)}.
\end{align}

Since we are working over the DVR $\RRR_\p$, the invariant factors $\pi_\p^{m_i}$ of the matrix $A$ must coincide with the invariant factors $\pi_\p^{n_i}$ of its transpose $A^T$. Equations (\ref{eq:disc_inv}), (\ref{eq:com_h0}) and (\ref{eq:tor_inv}) along with these observations let us deduce the following isomorphism of $\RRR_\p$-modules:
\begin{align*}
 H^0\left(\Q_\Sigma/\Q , \D^\star_{\varrho}\right)^\vee \otimes_\RRR \RRR_\p \cong H^1\left(\Q_\Sigma/\Q, \LLL^\star_{\varrho}\right)_{\RRR-\mathrm{tor}} \otimes_{\RRR} \RRR_\p.
\end{align*}
This completes the proof of the lemma.
\end{proof}

\begin{lemma} \label{lem:monic_prin}
Let $\p$ be a height one prime ideal in the power series ring $\RRR[[x]]$ containing a monic polynomial $h(x)$. Then, $\p$ is principal.
\end{lemma}

\begin{proof}
Let $\p_0$ denote the prime ideal $\p \cap \RRR[x]$. Under the ring extension $\RRR[x] \rightarrow \RRR[[x]]$, the prime ideal $\p$ is the extension of $\p_0$ (that is, $\p = \p_0\RRR[[x]]$). To see this, we can use the Euclidean algorithm pertaining to the Weierstrass preparation theorem \cite[Theorem 2.1 in Chapter 5]{MR1029028}. For every element $\beta$ in $\p$, there exists a polynomial $g(x)$ in $\RRR[x]$ with degree less than $\deg(h(x))$, and an element $q$ in $\RRR[[x]]$, satisfying $g(x) = \beta - q \cdot h(x)$. To prove the lemma, it suffices to show that the prime ideal $\p_0$ of $\R[x]$ is principal.

The ring $\RRR[[x]]$ is the completion of $\RRR[x]$ with respect of the ideal $(x)$. The extension $\RRR[x] \rightarrow \RRR[[x]]$ is a flat extension of Noetherian rings (\cite[Proposition 10.14]{atiyah1969introduction}). So, by the going down theorem (\cite[Lemma 10.11]{eisenbud1995commutative}), the height of the prime ideal $\p_0$ is also $1$.

Let $h_{\min}(x)$ denote the monic polynomial with the least degree amongst all the monic polynomials in $\p_0$. We claim that $h_{min}(x)$ generates $\p_0$. An application of Gauss' lemma (\cite[Corollary 4.12]{eisenbud1995commutative}) tells us that to show this, since the height of $\p_0$ equals one, it suffices to show that the element $h_{min}(x)$ is irreducible in $\RRR[x]$. Suppose we have $h_{min}(x) = h_1(x)h_2(x)$. The leading terms of $h_1(x)$ and $h_2(x)$ must be units in $\RRR$ since their product equals $1$. So, without loss of generality, assume that the polynomials $h_1(x)$ and $h_2(x)$ are also monic. Since $\p_0$ is prime, we may also suppose that $h_1(x)$ belongs to $\p_0$. Since $h_{\min}(x)$ is chosen to have the least degree amongst all the monic polynomials in $\p_0$, we have $\deg(h_1(x)) = \deg(h_{\min}(x))$ and $\deg(h_2(x))=0$. As a result, the monic polynomial $h_2(x)$ equals $1$ and $h_1(x)$ equals $h_{\min}(x)$. This shows that $h_{min}(x)$ is an irreducible element in $\RRR[x]$ and hence generates $\p_0$. This completes the proof of the lemma.
\end{proof}

\subsection{Cyclotomic twist deformations} \label{sec:cyc_twist}
Let $\RRR[[\Gamma_\Cyc]]$ denote the Iwasawa algebra $\varprojlim \limits_n \RRR[\Gamma_\Cyc / \Gamma_\Cyc^{p^n}]$. We shall follow all the notations used at the start of Section \ref{sec:selmer_struc}. To a Galois representation $\varrho : G_\Sigma \rightarrow \Gl_d(\RRR)$ satisfying the conditions (\ref{eq:fil-rho}) and (\ref{p-critical}), we shall associate the Galois representation $\varrho \otimes \kappa^{-1}: G_\Sigma \rightarrow \Gl_d(\RRR[[\Gamma_\Cyc]])$ given by the action of $G_\Sigma$ on $L_\varrho \otimes_\RRR \RRR[[\Gamma_\Cyc]](\kappa^{-1})$.  The Galois representation $\varrho \otimes \kappa^{-1}$ is called the cyclotomic twist deformation\footnote{Our definition of cyclotomic twist deformations differs from Greenberg's in that we consider twists by $\kappa^{-1}$ whereas Greenberg considers twists by $\kappa$. This distinction is not serious since one can translate results for cyclotomic twist deformations in our context to Greenberg's simply by considering the $\RRR$-linear ring isomorphism $\RRR[[\Gamma_\Cyc]] \cong \RRR[[\Gamma_\Cyc]]$ obtained by the involution sending $\gamma \rightarrow \gamma^{-1}$ for every $\gamma$ in $\Gamma_\Cyc$.} of $\varrho$. It is straightforward to see that $\varrho \otimes \kappa^{-1}$ also satisfies the condition (\ref{p-critical}).  One can also naturally associate to $\varrho\otimes \kappa^{-1}$, the following $\Gal{\overline{\Q}_p}{\Q_p}$-equivariant short exact sequence of free $\RRR[[\Gamma_\Cyc]]$-modules:
\begin{align}\label{eq:filtration}
0\rightarrow \Fil\LLL_{\varrho \otimes \kappa^{-1}} \rightarrow \LLL_{\varrho \otimes \kappa^{-1}} \rightarrow \frac{\LLL_{\varrho \otimes \kappa^{-1}}}{\Fil\LLL_{\varrho \otimes \kappa^{-1}}} \rightarrow 0.
\end{align}

Here, $\Fil\LLL_{\varrho \otimes \kappa^{-1}}$ denotes the $\RRR[[\Gamma_\Cyc]]$-module $\Fil\LLL_{\varrho } \otimes_\RRR \RRR[[\Gamma_\Cyc]](\kappa^{-1})$. We shall use the cumbersome notation $\varrho\otimes \kappa^{-1}$ to emphasize that the Galois representation is a cyclotomic twist deformation. Note that the residual representations associated to $\varrho$ and $\varrho \otimes \kappa^{-1}$ are isomorphic. \\

We shall define a discrete primitive Selmer group $\Sel\left(\Q,D_{\varrho\otimes \kappa^{-1}}\right)$ and a discrete ``strict'' Selmer group $\Sel^{\str}\left(\Q,\D_{\varrho\otimes \kappa^{-1}}\right)$ as follows:
\begin{align*}
  \Sel\left(\Q,\D_{\varrho\otimes \kappa^{-1}}\right):=\ker\bigg(H^1\left(G_\Sigma, \D_{\varrho\otimes \kappa^{-1}}\right) & \xrightarrow {\phi}  \frac{H^1\left(\Q_p,\D_{\varrho}\right)}{\Loc\left(\Q_p,D_{\varrho\otimes \kappa^{-1}}\right)} \oplus \bigoplus_{l \in \Sigma_0} H^1\left(\Q_l,\D_{\varrho\otimes \kappa^{-1}}\right)  \bigg), \\
    \Sel^{\str}\left(\Q,\D_{\varrho\otimes \kappa^{-1}}\right):=\ker\bigg(H^1\left(G_\Sigma, \D_{\varrho\otimes \kappa^{-1}}\right) & \rightarrow   \frac{H^1\left(\Q_p,\D_{\varrho\otimes \kappa^{-1}}\right)}{\Loc_{\mathrm{str}}\left(\Q_p,\D_{\varrho\otimes \kappa^{-1}}\right)} \oplus \bigoplus_{l \in \Sigma_0} H^1\left(\Q_l,\D_{\varrho\otimes \kappa^{-1}}\right)  \bigg).
\end{align*}
Here, the local conditions at $p$ are given below:
\begin{align*}
\Loc(\Q_p,\D_{\varrho\otimes \kappa^{-1}})&:=\ker\left(H^1\left(\Q_p,\D_{\varrho\otimes \kappa^{-1}}\right) \rightarrow H^1\left(I_p,\frac{\D_{\varrho\otimes \kappa^{-1}}}{\Fil \D_{\varrho\otimes \kappa^{-1}}}\right) \right), \\ \Loc_{\mathrm{str}}(\Q_p,\D_{\varrho\otimes \kappa^{-1}})&:=\ker\left(H^1\left(\Q_p,\D_{\varrho\otimes \kappa^{-1}}\right) \rightarrow H^1\left(\Q_p,\frac{\D_{\varrho\otimes \kappa^{-1}}}{\Fil \D_{\varrho\otimes \kappa^{-1}}}\right) \right).
\end{align*}

Using  local Euler-Poincaire charateristics (Proposition 4.2 in \cite{MR2290593}) along with Proposition 4.1 in \cite{palvannan2016algebraic}, it is straightforward to deduce the following equalities of ranks:

\begin{align*}
\mathrm{Rank}_{\RRR[[\Gamma_\Cyc]]}\Loc(\Q_p,\D_{\varrho\otimes \kappa^{-1}})^\vee = \mathrm{Rank}_{\RRR[[\Gamma_\Cyc]]}\Loc_{\mathrm{str}}(\Q_p,\D_{\varrho\otimes \kappa^{-1}})^\vee = d^+.
\end{align*}

\begin{remark}
Greenberg's definition of (primitive and strict) Selmer groups requires the global cocycles to be unramified at primes $l \in \Sigma_0$, whereas we require global cocycles to be ``locally trivial'' at all primes $l \in \Sigma_0$.  Since in this section we are only concerned with Galois representations that are cyclotomic twist deformations, this distinction does not matter since the natural map $H^1\left(\Q_l,\D_{\varrho\otimes \kappa^{-1}}\right) \rightarrow H^1\left(I_l,\D_{\varrho\otimes \kappa^{-1}}\right)$ turns out to be injective for all $l \in \Sigma_0$. Here, $I_l$ denotes the inertia subgroup inside $\Gal{\overline{\Q}_l}{\Q_l}$.
\end{remark}

\begin{remark}
When the Galois representation $\varrho$ satisfies Greenberg's Panchishkin condition, the cyclotomic deformation $\varrho \otimes \kappa^{-1}$ also satisfies Greenberg's Panchishkin condition. See \cite[Section 3]{greenberg1994iwasawa}.
\end{remark}

\begin{proposition}\label{prop:str-prim}
Suppose that the trivial representation is not a Jordan-Holder component of the residual representation $\overline{\varrho}$ for the action of $\Gal{\overline{\Q}_p}{\Q_p}$. Then, the natural inclusion $\Loc_{\mathrm{str}}(\Q_p,\D_{\varrho\otimes \kappa^{-1}}) \stackrel{\cong}{\hookrightarrow} \Loc(\Q_p,\D_{\varrho\otimes \kappa^{-1}})$ of discrete $\RRR[[\Gamma_\Cyc]]$-modules is an equality. Consequently, the natural inclusion $\Sel^{\str}\left(\Q,\D_{\varrho\otimes \kappa^{-1}}\right) \stackrel{\cong}{\hookrightarrow} \Sel\left(\Q,\D_{\varrho\otimes \kappa^{-1}}\right)$ of discrete $\RRR[[\Gamma_\Cyc]]$-modules is also an equality.
\end{proposition}

\begin{proof}
To prove the proposition, it suffices to show that the natural restriction map $$H^1\left(\Q_p,\frac{\D_{\varrho\otimes \kappa^{-1}}}{\Fil \D_{\varrho\otimes \kappa^{-1}}}\right) \xrightarrow {\Res} H^1\left(I_p,\frac{\D_{\varrho\otimes \kappa^{-1}}}{\Fil \D_{\varrho\otimes \kappa^{-1}}}\right)$$ is injective. By the inflation-restriction exact sequence, the kernel of this restriction map is naturally isomorphic to $H^1\left(\Gamma_p,H^0\left(I_p,\frac{\D_{\varrho\otimes \kappa^{-1}}}{\Fil \D_{\varrho\otimes \kappa^{-1}}}\right)\right)$. Here, $\Gamma_p$ denotes the quotient $\frac{\Gal{\overline{\Q}_p}{\Q_p}}{I_p}$. Note that $\Gamma_p$ is topologically generated by $\Frob_p$, the arithmetic Frobenius at $p$. \\

The hypothesis of the proposition and Corollary 3.1.1 in \cite{MR2290593}  let us deduce the following equality of discrete $\RRR[[\Gamma_\Cyc]]$-modules
\begin{align*}
0 = H^0\left(\Gal{\overline{\Q}_p}{\Q_p}, \frac{\D_{\varrho\otimes \kappa^{-1}}}{\Fil \D_{\varrho\otimes \kappa^{-1}}} \right) & \cong H^0\left(\Gamma_p, H^0\left(I_p, \frac{\D_{\varrho\otimes \kappa^{-1}}}{\Fil \D_{\varrho\otimes \kappa^{-1}}} \right)\right) \\ &\cong  H^0\left(I_p, \frac{\D_{\varrho\otimes \kappa^{-1}}}{\Fil \D_{\varrho\otimes \kappa^{-1}}} \right)[\Frob_p-1].
\end{align*}
So, we have the following exact sequence of finitely generated $\RRR[[\Gamma_\Cyc]]$-modules:
\begin{align*}
0 \rightarrow H^1\left(\Gamma_p, H^0\left(I_p, \frac{\D_{\varrho\otimes \kappa^{-1}}}{\Fil \D_{\varrho\otimes \kappa^{-1}}} \right)\right)^\vee  \rightarrow H^0\left(I_p, \frac{\D_{\varrho\otimes \kappa^{-1}}}{\Fil \D_{\varrho\otimes \kappa^{-1}}} \right)^\vee \xrightarrow {\Frob_p-1} H^0\left(I_p, \frac{\D_{\varrho\otimes \kappa^{-1}}}{\Fil \D_{\varrho\otimes \kappa^{-1}}} \right)^\vee \rightarrow 0.
\end{align*}
A straightforward application of Nakayama's Lemma tells us that a surjective endomorphism of a finitely generated module over a Noetherian ring is in fact an isomorphism (see Proposition 1.2 in \cite{MR0238839}). This lets us conclude that $H^1\left(\Gamma_p, H^0\left(I_p, \frac{\D_{\varrho\otimes \kappa^{-1}}}{\Fil \D_{\varrho\otimes \kappa^{-1}}} \right)\right)$ equals zero. The proposition follows.
\end{proof}

\subsection{Proof of Theorem \ref{theorem:second}}

\begin{theorem}
Suppose that the $\RRR$-module $\Sel(\Q,\D_{\varrho\otimes\kappa^{-1}})^\vee$ is torsion.  Then, the following divisor in $Z^1\left(\RRR[[\Gamma_\Cyc]]\right)$ is principal:
\begin{align*}
\Div\left(\Sel(\Q,\D_{\varrho\otimes\kappa^{-1}})^\vee\right) - \Div\left(H^0\left(\Q_\Sigma/\Q , \D_{\varrho\otimes\kappa^{-1}}\right)^\vee\right) - \Div\left(H^0\left(\Q_\Sigma/\Q , \D^\star_{\varrho\otimes\kappa^{-1}}\right)^\vee\right).
\end{align*}
\end{theorem}

\begin{proof}
A standard argument involving the Snake Lemma allows us to obtain the following exact sequence of $\RRR[[\Gamma_\Cyc]]$-modules relating the primitive Selmer group to the non-primitive Selmer group:
\begin{align} \label{es:imp}
0 \rightarrow \Sel(\Q,\D_{\varrho\otimes\kappa^{-1}}) \rightarrow \Sel^{\Sigma_0}(\Q,\D_{\varrho\otimes\kappa^{-1}}) \rightarrow \prod_{l \in \Sigma_0} H^1\left(\Q_l, \D_{\varrho\otimes\kappa^{-1}} \right) \rightarrow \coker(\phi) \rightarrow 0.
\end{align}
See Corollary 3.2.5 in Greenberg's work \cite{greenberg2010surjectivity}. Proposition 4.2 and Corollary 4.3 in \cite{palvannan2016algebraic} together let us deduce that for each $l \in \Sigma_0$, the $\RRR[[\Gamma_\Cyc]]$-module $H^1\left(\Q_l, \D_{\varrho\otimes\kappa^{-1}} \right)^\vee$ is torsion. These observations let us conclude that if the $\RRR[[\Gamma_\Cyc]]$-module $\Sel(\Q,\D_{\varrho\otimes\kappa^{-1}})^\vee$ is torsion, then the $\RRR[[\Gamma_\Cyc]]$-module $\Sel^{\Sigma_0}(\Q,\D_{\varrho\otimes\kappa^{-1}})^\vee$ is also torsion.
\begin{enumerate}
\item By Proposition 4.1 in \cite{palvannan2016algebraic}, the $\RRR$-modules $H^0\left(\Q_\Sigma/\Q, \D_{\varrho\otimes\kappa^{-1}}\right)^\vee$ and $H^0\left(I_p,\frac{\D_{\varrho\otimes\kappa^{-1}}}{\Fil \D_{\varrho\otimes\kappa^{-1}}}\right)^\vee$ are torsion.
\item By Proposition 4.6 in \cite{palvannan2016algebraic}, we have $H^2\left(\Q_\Sigma/\Q,\D_{\varrho\otimes\kappa^{-1}}\right)$ equals zero. Combining local duality and Proposition 4.1 in \cite{palvannan2016algebraic}, we have $H^2\left( \Q_p,\Fil \D_{\varrho\otimes\kappa^{-1}}\right) = H^2\left( \Q_p,\frac{\D_{\varrho\otimes\kappa^{-1}}}{\Fil \D_{\varrho\otimes\kappa^{-1}}}\right)=0$.
\item Similarly, combining local duality and Proposition 4.1 in \cite{palvannan2016algebraic} the third hypothesis of Proposition \ref{prop:prin_div} is also satisfied, that is $H^2\left(\Q_{l_0},\D_{\varrho\otimes\kappa^{-1}}\right)=0$. Note that $l_0$ is a finite prime in $\Sigma_0$.
\end{enumerate}
These observations verify all the conditions of Proposition \ref{prop:prin_div}. Applying Proposition \ref{prop:prin_div} lets us deduce that the following divisor in $Z^1\left(\RRR[[\Gamma_\Cyc]]\right)$ is principal:
\begin{align*}
\Div\left(\Sel^{\Sigma_0}(\Q,\D_{\varrho\otimes\kappa^{-1}})^\vee\right)  - \Div\left(H^0\left(\Q_\Sigma/\Q, \D_{\varrho\otimes\kappa^{-1}}\right)^\vee\right).
\end{align*}

One can use Proposition 8.6 in \cite{palvannan2016algebraic} to deduce that  for each $l \in \Sigma_0$, the divisor of the $\RRR[[\Gamma_\Cyc]]$-module $H^1\left(\Q_l, \D_{\varrho\otimes\kappa^{-1}} \right)^\vee$ is principal. As a result, the following divisor in $Z^1\left(\RRR[[\Gamma_\Cyc]]\right)$ is principal:
\begin{align} \label{eq:tempprin}
\Div\left(\Sel(\Q,\D_{\varrho\otimes\kappa^{-1}})^\vee\right) - \Div\left(H^0\left(\Q_\Sigma/\Q , \D_{\varrho\otimes\kappa^{-1}}\right)^\vee\right) - \Div\left(\coker(\phi)^\vee\right).
\end{align}

Let us fix an identification  of $\RRR[[\Gamma_\Cyc]]$ (as a complete local Noetherian ring) with the power series ring $\RRR[[x]]$ by sending a topological generator $\gamma_0$ of $\Gamma_\Cyc$ to $x+1$. By \cite[Proposition 4.1]{palvannan2016algebraic} and an application of the Cayley-Hamilton theorem \cite[Theorem 4.3]{eisenbud1995commutative}, the $\RRR[[\Gamma_\Cyc]]$-modules $H^0\left(\Q_\Sigma/\Q , \D_{\varrho\otimes\kappa^{-1}}\right)^\vee$ and $H^0\left(\Q_\Sigma/\Q , \D^\star_{\varrho\otimes\kappa^{-1}}\right)^\vee$ are annihilated by a monic polynomial. Applying Lemma \ref{lem:monic_prin} then lets us conclude the divisors $\Div\left(H^0\left(\Q_\Sigma/\Q , \D_{\varrho\otimes\kappa^{-1}}\right)^\vee\right)$ and $\Div\left(H^0\left(\Q_\Sigma/\Q , \D^\star_{\varrho\otimes\kappa^{-1}}\right)^\vee\right)$ are principal (since every height one prime ideal in the support of these divisors are principal). Using Poitou-Tate duality (for e.g. see the arguments in \cite[Section 3]{greenberg2010surjectivity}), we have the following natural inclusion of $\RRR[[\Gamma_\Cyc]]$-modules:
\begin{align} \label{eq:surj_coker}
\coker(\phi)^\vee \subset H^1\left(\Q_\Sigma/\Q, \LLL^\star_{\varrho\otimes\kappa^{-1}}\right)_{\RRR[[\Gamma_\Cyc]]-\mathrm{tor}}.
\end{align}
The facts that the $\RRR[[\Gamma_\Cyc]]$-module $\Sel(\Q,\D_{\varrho\otimes\kappa^{-1}})^\vee$ is torsion and that the Weak Leopoldt conjecture is valid play a key role. Combining our observations in this paragraph along with equation (\ref{eq:surj_coker}) and Lemma \ref{lem:tor_h0dual} then lets us conclude that $\Div\left(\coker(\phi)^\vee\right)$ is also principal. Using equation (\ref{eq:tempprin}), we can now conclude that the divisors $\Div\left(\Sel(\Q,\D_{\varrho\otimes\kappa^{-1}})^\vee\right)$, $\Div\left(H^0\left(\Q_\Sigma/\Q, \D_{\varrho\otimes\kappa^{-1}}\right)^\vee\right)$ and $\Div\left(H^0\left(\Q_\Sigma/\Q , \D^\star_{\varrho\otimes\kappa^{-1}}\right)^\vee\right)$ are all principal. Combining these observations lets us conclude that the following divisor in $Z^1\left(\RRR[[\Gamma_\Cyc]]\right)$ is principal:
\begin{align*}
\Div\left(\Sel(\Q,\D_{\varrho\otimes\kappa^{-1}})^\vee\right) - \Div\left(H^0\left(\Q_\Sigma/\Q , \D_{\varrho\otimes\kappa^{-1}}\right)^\vee\right) - \Div\left(H^0\left(\Q_\Sigma/\Q , \D^\star_{\varrho\otimes\kappa^{-1}}\right)^\vee\right).
\end{align*}
This completes the proof of the Theorem.
\end{proof}

\section{Hida families} \label{sec:Hida}

The hypotheses \ref{hyp:residual-irr} and \ref{hyp:p-distinguished} let us deduce that the Galois representations $\rho_F$, $\rho_G$ and $\rho_H$ associated to the Hida families $F$, $G$ and $H$ respectively satisfy Greenberg's Panchishkin condition. We have a $\Gal{\overline{\Q}_p}{\Q_p}$-equivariant short exact sequence $0 \rightarrow \Fil L_F \rightarrow L_F \rightarrow \frac{L_F}{ \Fil L_F} \rightarrow 0$ of free $R_F$-modules satisfying the following properties:
\begin{itemize}
\item $\mathrm{Rank}_{R_F}\Fil L_F = \mathrm{Rank}_{R_F} \frac{L_F}{ \Fil L_F} =1$,
\item  The action of $\Gal{\overline{\Q}_p}{\Q_p}$ on $\Fil L_F$  is via a ramified character $\delta_F : \Gal{\overline{\Q}_p}{\Q_p} \rightarrow R_F^\times$.
\item The action of $\Gal{\overline{\Q}_p}{\Q_p}$ on $\frac{L_F}{\Fil L_F}$ is via an unramified character $\epsilon_F : \Gal{\overline{\Q}_p}{\Q_p} \rightarrow R_F^\times$ such that $\epsilon_F(\Frob_p) = a_p(F)$. Here, $\Frob_p$ denotes the arithmetic Frobenius at $p$ and $a_p(F)$ denotes the $p$-th Fourier coefficient of $F$.
\end{itemize}
One can similarly associate $\Gal{\overline{\Q}_p}{\Q_p}$-equivariant filtrations $0 \rightarrow \Fil L_G \rightarrow L_G \rightarrow \frac{L_G}{ \Fil L_G} \rightarrow 0$ and $0 \rightarrow \Fil L_H \rightarrow L_H \rightarrow \frac{L_H}{ \Fil L_H} \rightarrow 0$ of free $R_G$ and $R_H$ modules associated to the Hida families  $G$ and $H$ respectively. For more details, one can refer to the summary of Hida's works (\cite{hida1986galois}, \cite{MR868300}) given in Section 2 of the paper by Emerton, Pollack and Weston \cite{emerton2006variation}.

\subsection{Tensor product of two Hida families}

We have a 4-dimensional Galois representation $\rhoDN{4}{2}:G_\Sigma \rightarrow \Gl_4\left(R_{F,G}\right)$, given by the action of $G_\Sigma$ on $L_{\pmb{4},{2}}:=L_F \ \hotimes \  L_G  \ \hotimes \ O (\psi)$, a free $R_{F,G}$-module of rank four. The Galois representation $\rhoDN{4}{3}$ that is described in the introduction, is isomorphic to $\rhoDN{4}{2} \otimes \kappa^{-1}$, the cyclotomic twist deformation of $\rhoDN{4}{2}$. Note that in this case the dimension  $d^+$ of the $+1$ eigenspace for the action of complex conjugation is equal to $2$. We will define the Selmer group $\Sel_\One(\Q,D_{\pmb{4},3})$ when $F$ is the dominant Hida family. One can similarly define the Selmer group $\Sel_\Two(\Q,D_{\pmb{4},3})$ in the case when $G$ is the dominant Hida family.

To define the Selmer group in the case when $F$ is the dominant Hida family, we will consider the following $\Gal{\overline{\Q}_p}{\Q_p}$-equivariant short exact sequence of free $R_{F,G}$-modules:
\begin{align*}
0 \rightarrow \mathrm{Fil}_{\One} L_{\pmb{4},2}  \rightarrow  L_{\pmb{4},2} \rightarrow \frac{ L_{\pmb{4},2}}{ \mathrm{Fil}_{\One} L_{\pmb{4},2}} \rightarrow 0,
\end{align*}
where $\mathrm{Fil}_\One L_{\pmb{4},2}$ denotes $\Fil L_F \ \hotimes \ L_G  \ \hotimes \ O(\psi)$, a free $R_{F,G}$-module of rank $2$. Note that in this case the condition (\ref{p-critical}) is satisfied since $\mathrm{Rank}_{R_{F,G}} \mathrm{Fil}_{\One} L_{\pmb{4},2} = d^+ = 2$.

The intersection $\mathrm{Fil}_{\mathrm{int}}L_{\pmb{4},2}:= \mathrm{Fil}_\One L_{\pmb{4},2} \bigcap \mathrm{Fil}_\Two L_{\pmb{4},2}$  equals $\Fil L_F \ \hotimes \ \Fil L_G  \ \hotimes \ O(\psi)$, a  free $R_{F,G}$-module of rank $1$.  It is also straightforward to check that the sum $\mathrm{Fil}_{\mathrm{sum}}L_{\pmb{4},2}:= \mathrm{Fil}_\One L_{\pmb{4},2} + \mathrm{Fil}_\Two L_{\pmb{4},2}$  is a  free $R_{F,G}$-module of rank $3$. Another straightforward computation lets us deduce that we have the following $\Gal{\overline{\Q}_p}{\Q_p}$-equivariant short exact sequences of free $R_{F,G}$-modules:
\begin{align*}
0 \rightarrow \mathrm{Fil}_{\mathrm{int}}L_{\pmb{4},2} \rightarrow  L_{\pmb{4},2} \rightarrow \frac{L_{\pmb{4},2}}{\mathrm{Fil}_{\mathrm{int}}L_{\pmb{4},2}} \rightarrow 0, \quad 0 \rightarrow \mathrm{Fil}_{\mathrm{sum}}L_{\pmb{4},2} \rightarrow  L_{\pmb{4},2} \rightarrow \frac{L_{\pmb{4},2}}{\mathrm{Fil}_{\mathrm{sum}}L_{\pmb{4},2}} \rightarrow 0
 \end{align*}

\begin{remark}
In the first case, the critical set of specializations $C_\One \subset \Hom_\cont\left(R_{F,G}[[\Gamma_\Cyc]],\overline{\Q}_p\right)$ is the set of continuous ring homorphisms $\phi: R_{F,G}[[\Gamma_\Cyc]] \rightarrow \overline{\Q}_p$ uniquely characterized by the following maps:
\begin{align} \label{eq:maps}
\phi_{F,k}: R_F \rightarrow  \overline{\Q}_p, \qquad \phi_{G,l}: R_F \rightarrow \overline{\Q}_p, \qquad \epsilon_p^{1-n}: \Gamma_\Cyc \rightarrow \overline{\Q}_p^\times,
\end{align}
where $\phi_{F,k}$ and $\phi_{G,l}$ vary over classical (weight $k$ and weight $l$ respectively) specializations of $F$ and $G$ respectively subject to the following condition:
\begin{align*}
\mathrm{Weight}(\phi_{F,k}) -1 \geq  n \geq \mathrm{Weight}(\phi_{G,l}) \geq 2.
\end{align*}
Here, $\epsilon_p : \Gamma_\Cyc \cong 1 + p\Z_p \hookrightarrow \overline{\Q}_p^\times$ is the canonical map sending the Frobenius at all primes $l \neq p$ to $ l /\omega(l)$ in $1+p\Z_p$.

In the second case, the critical set of specializations $C_\Two \subset \Hom_\cont\left(R_{F,G}[[\Gamma_\Cyc]],\overline{\Q}_p\right)$ is the set of continuous ring homorphisms $\phi: R_{F,G}[[\Gamma_\Cyc]] \rightarrow \overline{\Q}_p$ uniquely characterized by maps in equation (\ref{eq:maps}) subject to the following condition:
\begin{align*}
\mathrm{Weight}(\phi_{G,l}) -1  \geq  n \geq \mathrm{Weight}(\phi_{F,k}) \geq 2.
\end{align*}
\end{remark}

\subsection{Tensor product of three Hida families}

We have an 8-dimensional Galois representation $\rhoDN{8}{3}:G_\Sigma \rightarrow \Gl_8\left(R_{F,G,H}\right)$, given by the action of $G_\Sigma$ on $L_{\pmb{8},{3}}:=L_F \ \hotimes \  L_G\  \hotimes L_H \ \hotimes \ O (\psi)$, a free $R_{F,G,H}$-module of rank eight. The Galois representation $\rhoDN{8}{4}$ that is described in the introduction, is isomorphic to $\rhoDN{8}{3} \otimes \kappa^{-1}$, the cyclotomic twist deformation of $\rhoDN{8}{3}$. Note that in this case the dimension $d^+$ of the $+1$ eigenspace for the action of complex conjugation is equal to $4$.

\subsubsection{The unbalanced case ($F$ is the dominant Hida family)}

To define the Selmer group in the unbalanced case (when $F$ is the dominant Hida family), we will consider the following $\Gal{\overline{\Q}_p}{\Q_p}$-equivariant short exact sequence of free $R_{F,G,H}$-modules:
\begin{align*}
0 \rightarrow \mathrm{Fil}^{\unb} L_{\pmb{8},3}  \rightarrow  L_{\pmb{8},3} \rightarrow \frac{ L_{\pmb{8},3}}{ \mathrm{Fil}^{\unb} L_{\pmb{8},3}} \rightarrow 0,
\end{align*}
where $\mathrm{Fil}^{\unb} L_{\pmb{8},3}$ denotes $\Fil L_F \ \hotimes \ L_G  \ \hotimes \ L_H \ \hotimes \ O(\psi)$, a free $R_{F,G,H}$-module  of rank $4$. Note that in this case the condition (\ref{p-critical}) is satisfied since $\mathrm{Rank}_{R_{F,G,H}} \mathrm{Fil}^{\unb} L_{\pmb{8},3} = d^+ = 4$.  We let $\Sel_\One(\Q,D_{\pmb{4},3})$ denote the corresponding discrete Selmer group associated to $\rhoDN{8}{4}$.

\subsubsection{The balanced case}

To define the Selmer group in the balanced case, we will consider the following $\Gal{\overline{\Q}_p}{\Q_p}$-equivariant short exact sequence of free $R_{F,G,H}$-modules:
\begin{align*}
0 \rightarrow \mathrm{Fil}^{\bal} L_{\pmb{8},3}  \rightarrow  L_{\pmb{8},3} \rightarrow \frac{ L_{\pmb{8},3}}{ \mathrm{Fil}^{\bal} L_{\pmb{8},3}} \rightarrow 0,
\end{align*}
where $\mathrm{Fil}^{\bal} L_{\pmb{8},3}$ is defined to be the following $R_{F,G,H}$-module of rank $4$:
{\begin{align*}
\bigg(\Fil L_F \ \hotimes \ \Fil L_G  \ \hotimes \ L_H \ \hotimes \ O(\psi)\bigg)  & + \bigg(\Fil L_F \ \hotimes \ L_G  \ \hotimes \ \Fil L_H \ \hotimes \ O(\psi)\bigg)  \\ & +  \bigg(L_F \ \hotimes \ \Fil L_G  \ \hotimes \ \Fil L_H \ \hotimes \ O(\psi)\bigg).
\end{align*}

Note that in this case the condition (\ref{p-critical}) is satisfied since $ \mathrm{Rank}_{R_{F,G,H}} \mathrm{Fil}^{\bal} L_{\pmb{8},3} = d^+ = 4$. We let $\Sel_\Two(\Q,D_{\pmb{8},4})$ denote the corresponding discrete Selmer group associated to $\rhoDN{8}{4}$.

The intersection $\mathrm{Fil}_{\mathrm{int}} L_{\pmb{8},3}:= \mathrm{Fil}^{\unb} L_{\pmb{8},3} \bigcap  \mathrm{Fil}^{\bal} L_{\pmb{8},3}$  equals the following free $R_{F,G,H}$-module of rank $3$:
\begin{align} \label{eq:int_ten3}
\bigg(\Fil L_F \ \hotimes \ \Fil L_G  \ \hotimes \ L_H \ \hotimes \ O(\psi)\bigg)  + \bigg(\Fil L_F \ \hotimes \ L_G  \ \hotimes \ \Fil L_H \ \hotimes \ O(\psi)\bigg) .
\end{align}

It is also straightforward to check that the sum $\mathrm{Fil}_{\mathrm{sum}}L_{\pmb{4},2}:= \mathrm{Fil}^{\unb} L_{\pmb{8},3} +  \mathrm{Fil}^{\bal} L_{\pmb{8},3}$  is a  free $R_{F,G,H}$-module of rank $5$. Another straightforward computation lets us deduce that we have the following $\Gal{\overline{\Q}_p}{\Q_p}$-equivariant short exact sequences of free $R_{F,G,H}$-modules:
\begin{align*}
0 \rightarrow \mathrm{Fil}_{\mathrm{int}}L_{\pmb{8},3} \rightarrow L_{\pmb{8},3} \rightarrow \frac{L_{\pmb{8},3}}{\mathrm{Fil}_{\mathrm{int}}L_{\pmb{8},3}} \rightarrow 0, \quad 0 \rightarrow \mathrm{Fil}_{\mathrm{sum}}L_{\pmb{8},3} \rightarrow L_{\pmb{8},3} \rightarrow \frac{L_{\pmb{8},3}}{\mathrm{Fil}_{\mathrm{sum}}L_{\pmb{8},3}} \rightarrow 0
 \end{align*}

\begin{remark}
In the unbalanced case, the critical set of specializations $C_{\mathrm{unb}} \subset \Hom_\cont\left(R_{F,G,H}[[\Gamma_\Cyc]],\overline{\Q}_p\right)$ is the set of continuous ring homorphisms $\phi: R_{F,G,H}[[\Gamma_\Cyc]] \rightarrow \overline{\Q}_p$ uniquely characterized by the following maps:
\begin{align} \label{eq:maps_trip}
\phi_{F,k}: R_F \rightarrow  \overline{\Q}_p, \qquad \phi_{G,l}: R_F \rightarrow \overline{\Q}_p, \qquad \phi_{H,m}: R_F \rightarrow  \overline{\Q}_p, \qquad \epsilon_p^{1-n}: \Gamma_\Cyc \rightarrow \overline{\Q}_p^\times,
\end{align}
where $\phi_{F,k}$, $\phi_{G,l}$ and $\phi_{H,m}$ vary over classical (weight $k \geq 2$, weight $l \geq 2$ and weight $m \geq 2$) specializations of $F$, $G$ and $H$ respectively subject to the following condition:
\begin{align*}
\mathrm{Weight}(\phi_{F,k}) -1 \geq   n  \geq \mathrm{Weight}(\phi_{G,l}) + \mathrm{Weight}(\phi_{H,m}) -1.
\end{align*}

In the balanced case, the critical set of specializations $C_{\mathrm{bal}} \subset \Hom_\cont\left(R_{F,G,H}[[\Gamma_\Cyc]],\overline{\Q}_p\right)$ is the set of continuous ring homorphisms $\phi: R_{F,G,H}[[\Gamma_\Cyc]] \rightarrow \overline{\Q}_p$ uniquely characterized by the maps in equation (\ref{eq:maps_trip}) subject to all the following three conditions:
\begin{itemize}
\item $\mathrm{Weight}(\phi_{F,k})   \leq   n \leq \mathrm{Weight}(\phi_{G,l}) + \mathrm{Weight}(\phi_{H,m}) -2 $,
\item $\mathrm{Weight}(\phi_{G,l})   \leq   n \leq \mathrm{Weight}(\phi_{F,k}) + \mathrm{Weight}(\phi_{H,m}) - 2 $, and
\item  $\mathrm{Weight}(\phi_{H,m})   \leq   n  \leq \mathrm{Weight}(\phi_{F,k}) + \mathrm{Weight}(\phi_{G,l}) - 2$.
\end{itemize}
\end{remark}

\section{Description of $\cZ(\Q,D_{\pmb{d},n})$ and $\cZ(\Q,D^\star_{\pmb{d},n})$} \label{sec:desc}

In this section, we describe the construction of the modules $\cZ(\Q,D_{\pmb{d},n})$ and $\cZ(\Q,D^\star_{\pmb{d},n})$ that appear in the statement of Theorem \ref{maintheorem}. It may be helpful to point out that under the validity of the Iwasawa main conjectures, by \cite[Proposition 4.1]{lei2018codimension}, the hypothesis \ref{hyp:gcd} is equivalent to the condition that both $\cZ(\Q,D_{\pmb{d},n})$ and $\cZ(\Q,D^\star_{\pmb{d},n})$ are pseudo-null modules.

Let us first consider a general setting. Suppose to a Galois representation $\varrho : G_\Sigma \rightarrow \Gl_d(\RRR)$, one can associate two filtrations:
\begin{align*}
0\rightarrow \mathrm{Fil}_\One \LLL_\varrho \rightarrow \LLL_\varrho \rightarrow \frac{\LLL_\varrho}{\mathrm{Fil}_\One \LLL_\varrho} \rightarrow 0, \quad 0\rightarrow \mathrm{Fil}_\Two \LLL_\varrho \rightarrow \LLL_\varrho \rightarrow \frac{\LLL_\varrho}{\mathrm{Fil}_\Two \LLL_\varrho} \rightarrow 0,
\end{align*}
(that is, two $\Gal{\overline{\Q}_p}{\Q_p}$-equivariant short exact sequences of free $\RRR$-modules) satisfying the condition (\ref{p-critical}). We will consider the following $\RRR$-modules inside $\LLL_\varrho$:
\begin{align*}
\mathrm{Fil}_{\mathrm{int}} \LLL_\varrho := \mathrm{Fil}_\One \LLL_\varrho \bigcap \mathrm{Fil}_\Two, \LLL_\varrho, \qquad \mathrm{Fil}_{\mathrm{sum}} \LLL_\varrho := \mathrm{Fil}_\One \LLL_\varrho + \mathrm{Fil}_\Two, \LLL_\varrho,.
\end{align*}

Suppose also that all the $\RRR$-modules in the following short exact sequences are free:
\begin{align}\label{eq:sumint}
0 \rightarrow \mathrm{Fil}_{\mathrm{int}} \LLL_\varrho \rightarrow \LLL_\varrho \rightarrow \frac{\LLL_\varrho}{\mathrm{Fil}_{\mathrm{int}}\LLL_\varrho} \rightarrow 0, \quad 0 \rightarrow \mathrm{Fil}_{\mathrm{sum}}\LLL_\varrho \rightarrow \LLL_\varrho\rightarrow \frac{\LLL_\varrho}{\mathrm{Fil}_{\mathrm{sum}}\LLL_\varrho} \rightarrow 0.
 \end{align}

\begin{remark}
The existence of two filtrations satisfying the hypothesis in equation (\ref{eq:sumint}) for the Galois representations $\rho_{\pmb{4},{2}}$ and $\rho_{\pmb{8},3}$ has been discussed in Section \ref{sec:Hida}.
\end{remark}

Consider the following $\RRR[[\Gamma_\Cyc]]$-submodules of $\D_{\varrho\otimes \kappa^{-1}}$:
\begin{align*}
& \mathrm{Fil}_{\One}\D_{\varrho\otimes \kappa^{-1}} := \mathrm{Fil}_{\One} \LLL_\varrho \otimes_\RRR \RRR[[\Gamma_\Cyc]](\kappa^{-1}),  \qquad && \mathrm{Fil}_{\Two}\D_{\varrho\otimes \kappa^{-1}}:= \mathrm{Fil}_{\Two} \LLL_\varrho \otimes_\RRR \RRR[[\Gamma_\Cyc]](\kappa^{-1}), \\
& \mathrm{Fil}_{\mathrm{int}}\D_{\varrho\otimes \kappa^{-1}} := \mathrm{Fil}_{\mathrm{int}} \LLL_\varrho \otimes_\RRR \RRR[[\Gamma_\Cyc]](\kappa^{-1}),  \qquad && \mathrm{Fil}_{\mathrm{sum}}\D_{\varrho\otimes \kappa^{-1}} := \mathrm{Fil}_{\mathrm{sum}} \LLL_\varrho \otimes_\RRR \RRR[[\Gamma_\Cyc]](\kappa^{-1}).
\end{align*}
Note that we have the following equality inside $\D_{\varrho\otimes \kappa^{-1}}$:
\begin{align*}
\mathrm{Fil}_{\mathrm{int}}\D_{\varrho\otimes \kappa^{-1}} = \mathrm{Fil}_{\One} \D_{\varrho\otimes \kappa^{-1}} \bigcap \mathrm{Fil}_{\Two} \D_{\varrho\otimes \kappa^{-1}} , \qquad \mathrm{Fil}_{\mathrm{sum}}\D_{\varrho\otimes \kappa^{-1}} = \mathrm{Fil}_{\One} \D_{\varrho\otimes \kappa^{-1}} + \mathrm{Fil}_{\Two} \D_{\varrho\otimes \kappa^{-1}}.
\end{align*}

Let $\bullet$ denote an element in the set $\{\One, \Two, \mathrm{int}, \mathrm{sum}\}$. We define the following local conditions inside $H^1\left(\Q_p,D_{\varrho\otimes \kappa^{-1}}\right)$:
\begin{align*}
\Loc_{\bullet}(\Qp,\D_{\varrho\otimes \kappa^{-1}})&:=\ker\left(H^1\left(\Q_p,\D_{\varrho\otimes \kappa^{-1}}\right) \rightarrow H^1\left(\Q_p,\frac{\D_{\varrho\otimes \kappa^{-1}}}{\mathrm{Fil}_{\bullet}\D_{\varrho\otimes \kappa^{-1}}}\right) \right).
\end{align*}

Recall the following hypothesis from the introduction.

\begin{enumerate}[leftmargin=1.9cm, style=sameline, align=left, label={$\textsc{Loc}\mathrm{_p(0)}$ \textemdash}, ref={$\textsc{Loc}\mathrm{_p(0)}$}]
\item Neither the trivial representation nor the Teichm\"uller character $\omega$ is a Jordan-Holder component of the residual representation $\overline{\rho}$ for the action of $\Gal{\overline{\Q}_p}{\Q_p}$.
\end{enumerate}

\begin{proposition}\label{prop:locfree}
Suppose that $\varrho$ satisfies \ref{hyp:locp0}. Then, the $\RRR[[\Gamma_\Cyc]]$-modules $\Loc_\bullet\left(\Q_p,\D_{\varrho\otimes \kappa^{-1}}\right)^\vee$ and $\left(\dfrac{H^1\left(\Q_p, \D_{\varrho\otimes \kappa^{-1}}\right)}{\Loc_\bullet\left(\Q_p,\D_{\varrho\otimes \kappa^{-1}}\right)}\right)^\vee$ are free.
\end{proposition}

\begin{proof}
Since the trivial representation is not a Jordan-Holder component of the residual representation $\overline{\varrho}$, Corollary 3.1.1 in \cite{MR2290593} lets us deduce the following equality of discrete $\RRR[[\Gamma_\Cyc]]$-modules:
\begin{align*}
H^0\left(\Q_p,\mathrm{Fil}_\bullet \D_{\varrho\otimes \kappa^{-1}} \right) &= H^0\left(\Q_p,\D_{\varrho\otimes \kappa^{-1}} \right) = H^0\left(\Q_p,\frac{\D_{\varrho\otimes \kappa^{-1}}}{\mathrm{Fil}_\bullet \D_{\varrho\otimes \kappa^{-1}}} \right) = 0,
\end{align*}
Local Tate duality and Proposition 4.1 in \cite{palvannan2016algebraic} let us deduce the following equality of discrete $\RRR[[\Gamma_\Cyc]]$-modules:
\begin{align*}
H^2\left(\Q_p,\mathrm{Fil}_\bullet \D_{\varrho\otimes \kappa^{-1}} \right) &= H^2\left(\Q_p,\D_{\varrho\otimes \kappa^{-1}} \right) = H^2\left(\Q_p,\frac{\D_{\varrho\otimes \kappa^{-1}}}{\mathrm{Fil}_\bullet D_{\varrho\otimes \kappa^{-1}}} \right) = 0.
\end{align*}
To the short exact sequence (\ref{eq:filtration}), we can consider the long exact sequence in continuous group cohomology, for the group $\Gal{\overline{\Q}_p}{\Q_p}$. We obtain the following short exact sequence of discrete $\RRR[[\Gamma_\Cyc]]$-modules:
\begin{align*}
0 \rightarrow H^1\left(\Q_p,\mathrm{Fil}_\bullet \D_{\varrho\otimes \kappa^{-1}} \right) \rightarrow  H^1 \left(\Q_p,\D_{\varrho\otimes \kappa^{-1}} \right) \rightarrow  H^1\left(\Q_p,\frac{\D_{\varrho\otimes \kappa^{-1}}}{\mathrm{Fil}_\bullet \D_{\varrho\otimes \kappa^{-1}}} \right) \rightarrow 0.
\end{align*}
As a result, the discrete $\RRR[[\Gamma_\Cyc]]$-module $\Loc_\bullet\left(\Q_p,\D_{\varrho\otimes \kappa^{-1}}\right)$ is isomorphic to $H^1\left(\Q_p,\mathrm{Fil}_\bullet \D_{\varrho\otimes \kappa^{-1}} \right)$ and discrete $\RRR[[\Gamma_\Cyc]]$-module $\dfrac{H^1 \left(\Q_p,\D_{\varrho\otimes \kappa^{-1}} \right)}{\Loc_\bullet\left(\Q_p,\D_{\varrho\otimes \kappa^{-1}}\right)}$ is isomorphic to $H^1\left(\Q_p,\frac{\D_{\varrho\otimes \kappa^{-1}}}{\mathrm{Fil}_\bullet \D_{\varrho\otimes \kappa^{-1}}} \right)$. Proposition 5.10 and Remark 5.10.1 in \cite{MR2290593} let us conclude that the $\RRR[[\Gamma_\Cyc]]$-module $H^1\left(\Q_p,\mathrm{Fil}_\bullet \D_{\varrho\otimes \kappa^{-1}} \right)^\vee$ and $H^1\left(\Q_p,\frac{\D_{\varrho\otimes \kappa^{-1}}}{\mathrm{Fil}_\bullet \D_{\varrho\otimes \kappa^{-1}}} \right)^\vee$ are free. The proposition follows.
\end{proof}

\subsubsection*{Description of $\cZ(\Q,D_{\pmb{d},n})$}

Follow all the notations and hypotheses of Theorem \ref{maintheorem}. Let $\left(\rho,R,d,n\right)$ equal $\left(\rhoDN{4}{3},R_{F,G}[[\Gamma_\Cyc]],4,3\right)$ or $\left(\rhoDN{8}{4},R_{F,G,H}[[\Gamma_\Cyc]],8,4\right)$. As in \cite{lei2018codimension}, the $R$-module $\cZ(\Q,D_{\pmb{d},n})$ is defined to be the Pontryagin dual of
\begin{align} \label{eq:Z1def}
 \ker\bigg(H^1\left(G_\Sigma,D_{\pmb{d},n}\right) & \rightarrow   \frac{H^1\left(\Q_p,D_{\pmb{d},n}\right)}{\Loc_{\One}\left(\Q_p,D_{\pmb{d},n}\right) \bigcap \Loc_{\Two}\left(\Q_p,D_{\pmb{d},n}\right)} \oplus \bigoplus_{l \in \Sigma_0} H^1\left(\Q_l,D_{\pmb{d},n}\right)  \bigg).
  \end{align}
That is, the $R$-module $\cZ(\Q,D_{\pmb{d},n})$ is isomorphic to the Pontryagin dual of the intersection $\Sel_{\One}(\Q,D_{\rhoDN{d}{n}}) \bigcap \Sel_{\Two}(\Q,D_{\rhoDN{d}{n}})$ inside $H^1\left(G_\Sigma, D_{\rhoDN{d}{n}} \right)$. We observe that $\cZ(\Q,D_{\pmb{d},n})$ fits into the following surjection of $R$-modules:
\begin{align*}
  \underbrace{H^1\left(G_\Sigma, D_{\rhoDN{d}{n}} \right)^\vee}_{\substack{\text{Conjecturally} \\ \text{has $R$-rank $d^-$}}}  \twoheadrightarrow \underbrace{\substack{\Sel_{\One}(\Q,D_{\rhoDN{d}{n}})^\vee \\ \Sel_{\Two}(\Q,D_{\rhoDN{d}{n}})^\vee}}_{\substack{\text{Conjecturally} \\ \text{ $R$-torsion}}} \twoheadrightarrow \underbrace{\ZZZ(\Q,D_{\rhoDN{d}{n}})}_{\substack{\text{Pseudo-null, assuming}\\ \text{the hypotheses in Theorem \ref{maintheorem}}}} \twoheadrightarrow \Sha^{1}\left(\Q,D_{\rhoDN{d}{n}}\right)^\vee.
\end{align*}

\begin{proposition} \label{prop:sumint}
Suppose that $\varrho$ satisfies \ref{hyp:locp0}. Then, we have the following equalities inside $H^1\left(\Q_p,\D_{\varrho\otimes \kappa^{-1}}\right)$:
\begin{align*}
\Loc_\One(\Qp,\D_{\varrho\otimes \kappa^{-1}})\cap \Loc_\Two(\Qp,\D_{\varrho\otimes \kappa^{-1}})=\Loc_\Int(\Qp,\D_{\varrho\otimes \kappa^{-1}}),\\
\Loc_\One(\Qp,\D_{\varrho\otimes \kappa^{-1}})+\Loc_\Two(\Qp,\D_{\varrho\otimes \kappa^{-1}})=\Loc_\Sum(\Qp,\D_{\varrho\otimes \kappa^{-1}}).
\end{align*}
\end{proposition}
\begin{proof}
To ease notation, let us write
\begin{align*}
 & D= \D_{\varrho\otimes \kappa^{-1}}, \ \
 A =\mathrm{Fil}_\One \D_{\varrho\otimes \kappa^{-1}}, \ \   B=\mathrm{Fil}_\Two \D_{\varrho\otimes \kappa^{-1}}, \ \
 \Loc_A =\Loc_\One(\Qp,\D_{\varrho\otimes \kappa^{-1}}), \ \   \Loc_B =\Loc_\Two(\Qp,D_{\varrho\otimes \kappa^{-1}}),  \\
& \Loc_{\mathrm{int}} =\Loc_\Int(\Qp,D_{\varrho\otimes \kappa^{-1}}), \ \    \Loc_{\mathrm{sum}}=\Loc_A+\Loc_B.
\end{align*}
Our goal is to prove that $\Loc_A\cap \Loc_B \stackrel{?}{=}\Loc_{\mathrm{int}}$ and $\Loc_A+\Loc_B\stackrel{?}{=}\Loc_{\mathrm{sum}}$. There is a short exact sequence
\begin{equation}\label{eq:ses}
0\rightarrow\frac{D}{A\cap B}\stackrel{\alpha}{\longrightarrow} \frac{D}{A}\oplus \frac{D}{B}\stackrel{\beta}{\longrightarrow} \frac{D}{A+B}\rightarrow 0,
\end{equation}
where $\alpha $ is given by $x\mapsto(x\mod A,x\mod B)$ and $\beta$ is given by $(x,y)\mapsto x-y\mod A+B$.

As in the proof of Proposition~\ref{prop:locfree}, the hypothesis \ref{hyp:locp0} and local Tate duality imply that the $H^0(\Qp,-)$ and $H^2(\Qp,-)$ of all the quotients in \eqref{eq:ses} vanish. This gives the following short exact sequence:
\[
0\rightarrow H^1\left(\Qp,\frac{D}{A\cap B}\right){\longrightarrow} H^1\left(\Qp,\frac{D}{A}\right)\oplus H^1\left(\Qp, \frac{D}{B}\right){\longrightarrow} H^1\left(\Qp,\frac{D}{A+B}\right)\rightarrow 0,
\]
which sits inside the following commutative diagram:
\[\xymatrix{
    0 \ar[r] & H^1(\Qp,D) \ar[d] \ar[r]        & H^1(\Qp,D) \oplus H^1(\Qp,D)\ar[d] \ar[r]   & H^1(\Qp,D) \ar[d] \ar[r]      & 0 \\
    0 \ar[r] & H^1\left(\Qp,\frac{D}{A\cap B}\right) \ar[r] & H^1\left(\Qp,\frac{D}{A}\right)\oplus H^1\left(\Qp, \frac{D}{B}\right)
  \ar[r] & H^1\left(\Qp,\frac{D}{A+B}\right) \ar[r] & 0.
    }\]
 Note that the vertical arrows are all surjective by the vanishing of $H^2(\Qp,-)$. Hence, the snake lemma gives the following short exact sequence:
 \begin{equation}\label{eq:sesloc}
 0\rightarrow \Loc_{\mathrm{int}}\stackrel{\alpha'}{\longrightarrow }\Loc_A\oplus \Loc_B\stackrel{\beta'}{\longrightarrow }\Loc_{\mathrm{sum}}\rightarrow 0.\end{equation}
There are natural inclusions $\Loc_{\mathrm{int}}\subset\Loc_A$ and  $\Loc_B\subset\Loc_{\mathrm{sum}}$,  and the morhpisms in \eqref{eq:sesloc} are given by $\alpha':x\mapsto (x,x)$ and $\beta':(x,y)\mapsto x-y$ as in \eqref{eq:ses}. In particular,
\begin{align*}
\textrm{Im} \alpha'=\left\{(x,x):x\in\Loc_{\mathrm{int}}\right\},\quad
\ker\beta'=\left\{(x,x):x\in \Loc_A\cap \Loc_B\right\},\quad
\textrm{Im} \beta' = \Loc_A+\Loc_B.
\end{align*}
Hence, the exactness of \eqref{eq:sesloc} implies that $\Loc_A\cap \Loc_B=\Loc_{\mathrm{int}}$ and $\Loc_A+\Loc_B=\Loc_{\mathrm{sum}}$, as required.
\end{proof}

Proposition \ref{prop:sumint} lets us deduce that $\cZ(\Q,D_{\pmb{d},n})$ is isomorphic to the Pontryagin dual of
\begin{align}\label{eq:Z2def}
 \ker\bigg(H^1\left(G_\Sigma,D_{\pmb{d},n}\right) & \rightarrow   \frac{H^1\left(\Q_p,D_{\pmb{d},n}\right)}{\Loc_{\Int}\left(\Q_p,D_{\pmb{d},n}\right)} \oplus \bigoplus_{l \in \Sigma_0} H^1\left(\Q_l,D_{\pmb{d},n}\right)  \bigg).
  \end{align}

\subsubsection*{Description of $\cZ(\Q,D^\star_{\pmb{d},n})$}

Let $(\varrho\otimes \kappa^{-1})^\star:G_\Sigma \rightarrow \Gl_d(\RRR[[\Gamma_\Cyc]])$ denote the Galois representation given by the action of $G_\Sigma$ on the free $\RRR[[\Gamma_\Cyc]]$-module
$$\LLL^\star_{\varrho\otimes \kappa^{-1}}:=\Hom_{\RRR[[\Gamma_\Cyc]]}\left(\LLL_{\varrho\otimes \kappa^{-1}}, \RRR[[\Gamma_\Cyc]](\chi_p)\right).$$
Given two filtrations on $\varrho$, as in the beginning of Section \ref{sec:desc} that satisfy the condition (\ref{p-critical}), we naturally obtain two filtrations on $(\varrho\otimes \kappa^{-1})^\star$ (that is, two short exact sequences of free $\RRR[[\Gamma_\Cyc]]$-modules that are $\Gal{\overline{\Q}_p}{\Q_p}$-equivariant):
\begin{align*}
0\rightarrow \mathrm{Fil}_\One \LLL^\star_{\varrho\otimes \kappa^{-1}} \rightarrow \LLL^\star_{\varrho\otimes \kappa^{-1}} \rightarrow \frac{\LLL^\star_{\varrho\otimes \kappa^{-1}}}{\mathrm{Fil}_\One \LLL^\star_{\varrho\otimes \kappa^{-1}}} \rightarrow 0, \qquad 0\rightarrow \mathrm{Fil}_\Two \LLL^\star_{\varrho\otimes \kappa^{-1}} \rightarrow \LLL^\star_{\varrho\otimes \kappa^{-1}} \rightarrow \frac{\LLL^\star_{\varrho\otimes \kappa^{-1}}}{\mathrm{Fil}_\Two \LLL^\star_{\varrho\otimes \kappa^{-1}}} \rightarrow 0,
\end{align*}

\begin{align*}
\text{where } \mathrm{Fil}_\One \LLL^\star_{\varrho\otimes \kappa^{-1}} := \Hom\left(\frac{\LLL_{\varrho\otimes \kappa^{-1}}}{\mathrm{Fil}_\One \LLL_{\varrho\otimes \kappa^{-1}}}, \RRR[[\Gamma_\Cyc]](\chi_p)\right), \quad \mathrm{Fil}_\Two \LLL^\star_{\varrho\otimes \kappa^{-1}} := \Hom\left(\frac{\LLL_{\varrho\otimes \kappa^{-1}}}{\mathrm{Fil}_\Two \LLL_{\varrho\otimes \kappa^{-1}}},\RRR[[\Gamma_\Cyc]](\chi_p)\right).
\end{align*}
Both these filtrations satsify the condition (\ref{p-critical}).  Suppose that we also have short exact sequences as in equation (\ref{eq:sumint}). One can then deduce that we have the following natural isomorphisms:
\begin{align*}
\mathrm{Fil}_\Int \LLL^\star_{\varrho\otimes \kappa^{-1}} &\cong \Hom_{\RRR[[\Gamma_\Cyc]]}\left(\frac{\LLL_{\varrho\otimes \kappa^{-1}}}{\mathrm{Fil}_\One \LLL_{\varrho\otimes \kappa^{-1}} + \mathrm{Fil}_\Two \LLL_{\varrho\otimes \kappa^{-1}}}, \RRR[[\Gamma_\Cyc]](\chi_p)\right), \\ \mathrm{Fil}_\Sum \LLL^\star_{\varrho\otimes \kappa^{-1}} &\cong \Hom_{\RRR[[\Gamma_\Cyc]]}\left(\frac{\LLL_{\varrho\otimes \kappa^{-1}}}{\mathrm{Fil}_\One \LLL_{\varrho\otimes \kappa^{-1}} \bigcap \mathrm{Fil}_\Two \LLL_{\varrho\otimes \kappa^{-1}}}, \RRR[[\Gamma_\Cyc]](\chi_p)\right),
\end{align*}
involving the $\RRR[[\Gamma_\Cyc]]$-modules $\mathrm{Fil}_\Int \LLL^\star_{\varrho\otimes \kappa^{-1}} := \mathrm{Fil}_\One \LLL^\star_{\varrho\otimes \kappa^{-1}} \cap \mathrm{Fil}_\Two \LLL^\star_{\varrho\otimes \kappa^{-1}}$ and $\mathrm{Fil}_\Sum \LLL^\star_{\varrho\otimes \kappa^{-1}}:= \mathrm{Fil}_\One \LLL^\star_{\varrho\otimes \kappa^{-1}}  + \mathrm{Fil}_\Two \LLL^\star_{\varrho\otimes \kappa^{-1}}$.

We also obtain the following short exact sequences of free $\RRR[[\Gamma_\Cyc]]$-modules, similar to equation (\ref{eq:sumint}), that are $\Gal{\overline{\Q}_p}{\Q_p}$-equivariant:
\begin{align}
& 0 \rightarrow \mathrm{Fil}_{\mathrm{int}} \LLL^\star_{\varrho\otimes \kappa^{-1}} \rightarrow \LLL^\star_{\varrho\otimes \kappa^{-1}}\rightarrow \frac{\LLL^\star_{\varrho\otimes \kappa^{-1}}}{\mathrm{Fil}_{\mathrm{int}}\LLL^\star_{\varrho\otimes \kappa^{-1}}} \rightarrow 0, \\ \notag &  0 \rightarrow \mathrm{Fil}_{\mathrm{sum}}\LLL^\star_{\varrho\otimes \kappa^{-1}} \rightarrow \LLL^\star_{\varrho\otimes \kappa^{-1}}\rightarrow \frac{\LLL^\star_\varrho}{\mathrm{Fil}_{\mathrm{sum}}\LLL^\star_{\varrho\otimes \kappa^{-1}}} \rightarrow 0.
 \end{align}

Let $\Loc_{\One}(\Q,\D^\star_{\varrho\otimes \kappa^{-1}})$, $\Loc_{\Two}(\Q,\D^\star_{\varrho\otimes \kappa^{-1}})$ and $\Loc_{\Int}(\Q,\D^\star_{\varrho\otimes \kappa^{-1}})$ denote the discrete local conditions inside $H^1\left(\Q_p,\D^\star_{\varrho\otimes \kappa^{-1}}\right)$ obtained from the filtrations $\mathrm{Fil}_\One \LLL^\star_{\varrho\otimes \kappa^{-1}}$, $\mathrm{Fil}_\Two \LLL^\star_{\varrho\otimes \kappa^{-1}}$ and $\mathrm{Fil}_\Int \LLL^\star_{\varrho\otimes \kappa^{-1}}$ respectively. Suppose that $\varrho$ satisfies \ref{hyp:locp0}. An analog of Proposition \ref{prop:sumint} would tell us that inside $H^1\left(\Q_p,\D^\star_{\varrho\otimes \kappa^{-1}}\right)$ the intersection $\Loc_{\One}(\Q,\D^\star_{\varrho\otimes \kappa^{-1}}) \bigcap \Loc_{\Two}(\Q,\D^\star_{\varrho\otimes \kappa^{-1}})$  is equal to $\Loc_{\Int}(\Q,\D^\star_{\varrho\otimes \kappa^{-1}})$. By Proposition \ref{prop:sumint}, we have the following equality of $\RRR[[\Gamma_\Cyc]]$-modules:
\begin{align*}
\Loc_\mathrm{sum} \left(\Q_p,\D_{\varrho \otimes \kappa^{-1}}\right) = \Loc_\One\left(\Q_p,\D_{\varrho \otimes \kappa^{-1}}\right)+ \Loc_\Two\left(\Q_p,\D_{\varrho \otimes \kappa^{-1}}\right).
\end{align*}
For an $R$-module $M$, we let $M^\dagger$ denote the reflexive dual $\Hom_R(M,R)$. Using the duality theorems in \cite{lei2018codimension} (see equation (3.19) in \cite{lei2018codimension}), we can deduce the following isomorphism of $\RRR[[\Gamma_\Cyc]]$-modules:
\begin{align*}
\bigg(\Loc_\mathrm{sum} \left(\Q_p,\D_{\varrho \otimes \kappa^{-1}}\right) ^\vee\bigg)^\dagger \cong H^1_\ct\left(\Q_p, \mathrm{Fil}_{\mathrm{sum}} \LLL_{\varrho \otimes \kappa^{-1}}\right) \subset H^1_\ct\left(\Q_p,  \LLL_{\varrho \otimes \kappa^{-1}}\right).
\end{align*}
We can now describe $\cZ(\Q,D^\star_{\pmb{d},n})$. Follow all the notations and hypotheses of Theorem \ref{maintheorem}.  Let $\left(\rho,R,d,n\right)$ equal $\left(\rhoDN{4}{3},R_{F,G}[[\Gamma_\Cyc]],4,3\right)$ or $\left(\rhoDN{8}{4},R_{F,G,H}[[\Gamma_\Cyc]],8,4\right)$.
 Let $L^\star_\perp$ denote the submodule of $H^1\left(\Q_p,D^\star_{\pmb{d},n}\right)$ that equals the orthogonal complement of $\left(\left(\Loc_\One\left(\Q_p,D_{\pmb{d},n}\right)+ \Loc_\Two\left(\Q_p,D_{\pmb{d},n}\right) \right)^\vee\right)^\dagger$ under the local pairing given in equation (\ref{eq:per_pairing}). Following \cite{lei2018codimension}, the $\RRR$-module $\cZ(\Q,D^\star_{\pmb{d},n})$ is defined to be the Pontryagin dual of
\begin{align}\label{eq:ZS1def}
 \ker\bigg(H^1\left(G_\Sigma, D^\star_{\pmb{d},n} \right) & \rightarrow   \frac{H^1\left(\Q_p,\ D^\star_{\pmb{d},n} \right)}{L^\star_\perp} \oplus \bigoplus_{l \in \Sigma_0} H^1\left(\Q_l, D^\star_{\pmb{d},n} \right)  \bigg).
  \end{align}

\begin{lemma} \label{lem:refpondual}
Suppose that $\varrho$ satisfies \ref{hyp:locp0}. Under the perfect pairing (given by local Tate duality):
\begin{align} \label{eq:per_pairing}
H^1\left(\Q_p,\D^\star_{\varrho \otimes \kappa^{-1}}\right) \times H^1_\ct\left(\Q_p, \LLL_{\varrho \otimes \kappa^{-1}}\right) \rightarrow \Q_p/\Z_p,
\end{align}
$\Loc_{\Int}\left(\Q_p,\D^\star_{\varrho \otimes \kappa^{-1}}\right)$ is the orthogonal complement of $\left(\left(\Loc_\One\left(\Q_p,\D_{\varrho \otimes \kappa^{-1}}\right)+ \Loc_\Two\left(\Q_p,\D_{\varrho \otimes \kappa^{-1}}\right) \right)^\vee\right)^\dagger$.
\end{lemma}

\begin{proof}

The lemma follows by applying local duality along with the hypothesis \ref{hyp:locp0} which then provides us the following commutative diagram:
\begin{align*}
\xymatrix{
0 \ar[r] & H^1_\ct\left(\Q_p, \mathrm{Fil}_{\mathrm{sum}} \LLL_{\varrho \otimes \kappa^{-1}}\right) \ar[d]^{\cong} \ar[r]& H^1_\ct\left(\Q_p, \LLL_{\varrho \otimes \kappa^{-1}}\right) \ar[d]^{\cong} \ar[r] & H^1_\ct\left(\Q_p, \frac{\LLL_{\varrho \otimes \kappa^{-1}}}{\mathrm{Fil}_{\mathrm{sum}} \LLL_{\varrho \otimes \kappa^{-1}}} \right) \ar[d]^{\cong} \ar[r]&  0,\\
0 \ar[r] & H^1\left(\Q_p, \frac{\D^\star_{\varrho \otimes \kappa^{-1}}}{\mathrm{Fil}_{\mathrm{int}} \D^\star_{\varrho \otimes \kappa^{-1}}}\right)^\vee  \ar[r]& H^1\left(\Q_p, \D^\star_{\varrho \otimes \kappa^{-1}}\right)^\vee \ar[r] & H^1 \left(\Q_p, \mathrm{Fil}_{\mathrm{int}} \D^\star_{\varrho \otimes \kappa^{-1}} \right)^\vee \ar[r]&  0.
}
\end{align*}
\end{proof}

Lemma \ref{lem:refpondual} lets us deduce that $\cZ(\Q,D^\star_{\pmb{d},n})$ is isomorphic to be the Pontryagin dual of
\begin{align}\label{eq:ZS3def}
& \ker\bigg(H^1\left(G_\Sigma, D^\star_{\pmb{d},n} \right) \rightarrow   \frac{H^1\left(\Q_p,\ D^\star_{\pmb{d},n} \right)}{\Loc_{\Int}\left(\Q_p, \ D^\star_{\pmb{d},n} \right)} \oplus \bigoplus_{l \in \Sigma_0} H^1\left(\Q_l, D^\star_{\pmb{d},n} \right)  \bigg) \\
\cong  \ &  \ker\bigg(H^1\left(G_\Sigma,D^\star_{\pmb{d},n}\right)  \rightarrow   \frac{H^1\left(\Q_p,D^\star_{\pmb{d},n}\right)}{\Loc_{\One}\left(\Q_p,D^\star_{\pmb{d},n}\right) \bigcap \Loc_{\Two}\left(\Q_p,D^\star_{\pmb{d},n}\right)} \oplus \bigoplus_{l \in \Sigma_0} H^1\left(\Q_l,D^\star_{\pmb{d},n}\right)  \bigg).
  \end{align}

Let $\Sel_{\One}(\Q,D^\star_{\rhoDN{d}{n}})$ and $\Sel_{\Two}(\Q,D^\star_{\rhoDN{d}{n}})$ denote the  discrete Selmer groups inside $H^1\left(G_\Sigma, D^\star_{\rhoDN{d}{n}} \right)$ defined using the local conditions $\Loc_{\One}(\Q_p,D^\star_{\rhoDN{d}{n}})$ and $\Loc_{\Two}(\Q_p,D^\star_{\rhoDN{d}{n}})$ respectively. Our observations tell us that $\cZ(\Q,D^\star_{\pmb{d},n})$ is isomorphic to the Pontryagin dual of the intersection $\Sel_{\One}(\Q,D^\star_{\rhoDN{d}{n}}) \bigcap \Sel_{\Two}(\Q,D^\star_{\rhoDN{d}{n}})$. Also, $\cZ(\Q,D^\star_{\pmb{d},n})$ fits into the following surjection of $R$-modules:
\begin{align*}
  \underbrace{H^1\left(G_\Sigma, D^\star_{\rhoDN{d}{n}} \right)^\vee}_{\substack{\text{Conjecturally} \\ \text{has $R$-rank $d^-(\rho^\star)$}}}  \twoheadrightarrow \underbrace{\substack{\Sel_{\One}(\Q,D^\star_{\rhoDN{d}{n}})^\vee \\ \Sel_{\Two}(\Q,D^\star_{\rhoDN{d}{n}})^\vee}}_{\substack{\text{Conjecturally} \\ \text{ $R$-torsion}}} \twoheadrightarrow \underbrace{\ZZZ(\Q,D^\star_{\rhoDN{d}{n}})}_{\substack{\text{Pseudo-null, assuming}\\ \text{the hypotheses in Theorem \ref{maintheorem}}}} \twoheadrightarrow \Sha^{1}\left(\Q,D^\star_{\rhoDN{d}{n}}\right)^\vee.
\end{align*}

\section{Proof of Theorem \ref{maintheorem}} \label{sec:proof_main}

Theorem \ref{maintheorem} directly follows from Theorem 4.3 in \cite{lei2018codimension}. We will need to verify all the hypotheses of the theorem there. \\

By Proposition \ref{prop:locfree}, the $R$-modules $\Loc_\One(\Q_p,D_{\pmb{d},{n}})^\vee$ and $\Loc_\Two(\Q_p,D_{\pmb{d},{n}})^\vee$ are free. Using local Euler-Poincair\'e characteristics, one can conclude that both these $R$-mdoules have rank $d^+$. By Proposition \ref{prop:sumint}, the $R$-module $\Loc_\One(\Q_p,D_{\pmb{d},{n}}) + \Loc_\Two(\Q_p,D_{\pmb{d},{n}})$ is equal to $\Loc_\Sum(\Q_p,D_{\pmb{d},{n}})$ and $\Loc_\One(\Q_p,D_{\pmb{d},{n}}) \bigcap \Loc_\Two(\Q_p,D_{\pmb{d},{n}})$ is equal to $\Loc_\Int(\Q_p,D_{\pmb{d},{n}})$.   Using Proposition \ref{prop:locfree} and local Euler-Poincair\'e characteristics, the $R$-module $\Loc_\Sum(\Q_p,D_{\pmb{d},{n}})^\vee$ is free of rank $d^++1$.  The short exact sequence of $R$-modules
$$ 0 \rightarrow \left(\dfrac{\Loc_{\Two}\left(\Q_p,D_{\rhodn}\right)}{\Loc_\Int(\Q_p,D_{\pmb{d},{n}})}\right)^\vee \rightarrow \Loc_\Sum(\Q_p,D_{\pmb{d},{n}})^\vee \rightarrow \Loc_{\One}\left(\Q_p,D_{\rhodn}\right)^\vee \rightarrow  0$$
and our earlier observations let us conclude that the $R$-module $\left(\dfrac{\Loc_{\Two}\left(\Q_p,D_{\rhodn}\right)}{\Loc_\Int(\Q_p,D_{\pmb{d},{n}})}\right)^\vee$ is also free.  A similar argument shows that the $R$-module $\left(\dfrac{\Loc_{\One}\left(\Q_p,D_{\rhodn}\right)}{\Loc_\Int(\Q_p,D_{\pmb{d},{n}})}\right)^\vee$ is also free.  The conditions labeled ``$\mathrm{Hypothesis \ Rank}$'' and ``$\mathrm{Hypothesis \ LF}$'' follow from these observations. {Note that a finitely generated $R_{F,G}[[\Gamma_\Cyc]]$ (or $R_{F,G,H}[[\Gamma_\Cyc]]$) module must be torsion if it is also finitely generated over the coefficient ring $R_{F,G}$ (or $R_{F,G,H}$ respectively).} The condition labeled ``$\mathrm{Loc}(0)$'' follows from this observation and Proposition 4.1 in \cite{palvannan2016algebraic}. The condition labeled ``$\mathrm{Reg}(0)$'' follows from Proposition 4.17 in \cite{palvannan2016algebraic}.

To establish the condition labeled ``$\mathrm{Gor}$'', we need to prove that the rings $R_{F,G}[[\Gamma_\Cyc]]$ and $R_{F,G,H}[[\Gamma_\Cyc]]$ are Gorenstein. We will show that $R_{F,G}[[\Gamma_\Cyc]]$ is Gorenstein. Similar arguments can be used to show that $R_{F,G,H}[[\Gamma_\Cyc]]$ is also Gorenstein. We shall fix a topological generator $\gamma_0$ of $\Gamma_\Cyc$, and hence fix an isomorphism $R_{F,G}[[\Gamma_\Cyc]]  \cong R_{F,G}[[s]]$ by sending $\gamma_0-1$ to $s$. The natural $O[[x_G]]$-inclusions $R_{F}[[x_G,s]] \hookrightarrow R_{F,G}[[\Gamma_\Cyc]]$ and $R_{G} \hookrightarrow R_{F,G}[[\Gamma_\Cyc]]$ give us a ring homomorphism $R_{F}[[x_G,s]] \otimes_{O[[x_G]]} R_G \twoheadrightarrow R_{F,G}[[\Gamma_\Cyc]] \cong R_{F,G}[[s]]$. This ring homomorphism must be surjective since this is a map of complete $O$-algebras whose image contains $R_F$, $R_G$ and $s$ (sitting naturally {inside} $R_{F,G}[[s]]$). The rings $R_F$ and $R_G$, being Noetherian normal local domains with Krull dimension $2$, are also Cohen-Macaulay. See \cite[Exercise 2.2.30(a)]{MR1251956}. By  \cite[Proposition 2.2.11]{MR1251956}, we can conclude that the rings $R_F$ and $R_G$ are finitely generated free $O[[x_F]]$ and $O[[x_G]]$-modules with ranks (say) $r_1$ and $r_2$ respectively. By \cite[Proposition 5.5.3]{ribes2000profinite}, $R_{F,G}[[s]]$ is a free $O[[x_F,x_G,s]]$-module with rank $r_1r_2$. Using Nakayama's lemma, we can conclude that $R_F[[x_G,s]]$ is a free $O[[x_F,x_G,s]]$-module with rank $r_1$. For example, see the proof of \cite[Lemma 4.3.13]{MR3076731}. As a result, the tensor product $R_{F}[[x_G,s]] \otimes_{O[[x_G]]} R_G$ is a free $O[[x_F,x_G,s]]$-module with rank $r_1r_2$. We must have the following isomorphism:
\begin{align}\label{eq:completetensor}
R_{F}[[x_G,s]] \otimes_{O[[x_G]]} R_G \xrightarrow {\cong} R_{F,G}[[\Gamma_\Cyc]],
\end{align}
since the natural map $R_{F}[[x_G,s]] \otimes_{O[[x_G]]} R_G \twoheadrightarrow R_{F,G}[[\Gamma_\Cyc]]$ is surjective and can be viewed as an $O[[x_F,x_G,s]]$-module homomorphism between two free modules  with the same rank.

The ring $R_G$ is a finitely generated $O[[x_G]]$-module. Since $R_F[[s]]$ is a torsion-free module over the DVR $O$, the extension $O \hookrightarrow R_F[[s]]$ is flat. By \cite[Theorem 22.3]{matsumura1989commutative}, the extension $O[[x_G]] \hookrightarrow R_{F}[[x_G,s]]$ is also flat. The hypothesis \ref{hyp:gor} and \cite[Corollary 3.3.21]{MR1251956} let us conclude that the rings $O[[x_G]]$, $R_{F}[[x_G,s]]$ and $R_G$ are Gorenstein. These observations, equation (\ref{eq:completetensor}) along with \cite[Theorem 2]{MR0257062} establish that the ring $R_{F,G}[[\Gamma_\Cyc]]$ is Gorenstein. \\

We need to further establish two conditions related to the $\RRR$-module labeled $\XXX(\Q,D_{\pmb{d},n})$ in \cite{lei2018codimension} and which is defined as the Pontryagin dual of:
\begin{align*}
 \ker\bigg(H^1\left(G_\Sigma, D_{\pmb{d},n} \right) & \xrightarrow {\phi_\XXX}  \frac{H^1\left(\Q_p,\ D_{\pmb{d},n}\right)}{\Loc_\One(\Q_p,D_{\pmb{d},n})+\Loc_\Two(\Q_p,D_{\pmb{d},n})} \oplus \bigoplus_{l \in \Sigma_0} H^1\left(\Q_l, D_{\pmb{d},n} \right)  \bigg).
  \end{align*}
For the first condition involving $\XXX(\Q,D_{\pmb{d},n})$, we need to establish that for every height two prime ideal $\QQQ$ in $R$, the $R_\QQQ$-module $\XXX(\Q,D_{\pmb{d},n})_\QQQ$ has finite projective dimension. By Proposition \ref{prop:sumint}, the discrete $R$-module $\Loc_\One(\Q_p,D_{\pmb{d},n})+\Loc_\Two(\Q_p,D_{\pmb{d},n})$ is equal to $\Loc_\Sum(\Q_p,D_{\pmb{d},n})$. By Proposition \ref{prop:locfree}, the $R$-module $\left(\dfrac{H^1\left(\Q_p,\ D_{\pmb{d},n}\right)}{\Loc_\Sum(\Q_p,D_{\pmb{d},n})}\right)^\vee$ is free. Combining Propositions 4.17 and 5.5 in \cite{palvannan2016algebraic}, one can conclude that for every height two prime ideal $\QQQ$ in $R$ and for every $l$ in $\Sigma_0$, the $R_\QQQ$-modules $\left(H^1\left(G_\Sigma, D_{\pmb{d},n} \right)^\vee\right)_\QQQ$ and $\left(H^1\left(\Q_l, D_{\pmb{d},n} \right)^\vee\right)_\QQQ$ have finite projective dimension. Using \cite[Lemma 3.6]{lei2018codimension}, one can conclude that the map $\phi_\XXX$ is surjective. These observations now let us conclude that for every height two prime ideal $\QQQ$ in $R$, the $R_\QQQ$-module $\XXX(\Q,D_{\pmb{d},n})_\QQQ$ has finite projective dimension.

Consider the reflexive hull $\XXX(\Q,D_{\pmb{d},n})^{\dagger \dagger}$ of the $R$-module $\XXX(\Q,D_{\pmb{d},n})$. Let $\QQQ$ be a height two prime ideal in $R$. For the second condition involving $\XXX(\Q,D_{\pmb{d},n})$, we need to establish that the $R_\QQQ$-module $\coker(\XXX(\Q,D_{\pmb{d},n}) \rightarrow \XXX(\Q,D_{\pmb{d},n})^{\dagger \dagger}) \otimes_{R} R_\QQQ$ has finite projective dimension.

By \cite[Corollary 4.2]{lei2018codimension}, the $R$-module $\XXX(\Q,D_{\pmb{d},n})$ is torsion-free. In the previous paragraph, we have shown that the projective dimension of the $R_\QQQ$-module $\XXX(\Q,D_{\pmb{d},n})_\QQQ$ is finite. It will thus be enough to establish that the $R$-module $\XXX(\Q,D_{\pmb{d},n})^{\dagger \dagger}$ is free. We will show that the $R$-module $\XXX(\Q,D_{\pmb{d},n})^{\dagger}$ is free. By applying the functor $\Hom_R(\mbox{--},R)$ to the short exact sequence in (\cite[Equation 3.10]{lei2018codimension}), we obtain the following exact sequence of $R$-modules:
\begin{align} \label{eq:Xdagger}
0 \rightarrow \XXX(\Q,D_{\pmb{d},n})^{\dagger} \rightarrow R \rightarrow  \Ext^1_{R}\left(\Sel_\One(\Q,D_{\pmb{d},n})^\vee,R\right) \rightarrow \Ext^1_{R}\left(\XXX(\Q,D_{\pmb{d},n}),R\right)
\end{align}

To obtain equation  (\ref{eq:Xdagger}) from (\cite[Equation 3.10]{lei2018codimension}), we have identified the free $R$-module $\left(\dfrac{\Loc_{\Two}\left(\Q_p,D_{\rhodn}\right)}{{\Loc_{\Int}\left(\Q_p,D_{\rhodn}\right)}}\right)^\vee$ of rank one with $R$. This allows us to identify $\XXX(\Q,D_{\pmb{d},n})^{\dagger}$ with an ideal inside $R$. This ideal must be reflexive over $R$. The $R$-module $ \XXX(\Q,D_{\pmb{d},n})$ is torsion-free. Since $R$ is integrally closed, for every height one prime ideal $\p$ in $R$, the localization $R_\p$ must be a DVR and consequently the $R_\p$-module $\XXX(\Q,D_{\pmb{d},n})_\p$ must be free. Since localization commutes with $\Ext$, the $R$-module $\Ext^1_{R}\left(\XXX(\Q,D_{\pmb{d},n}),R\right)$ must be pseudo-null. Using this trick of localizing at every height one prime ideal of $R$, we have
$\Div\bigg( \Ext^1_{R}\left(\Sel_\One(\Q,D_{\pmb{d},n})^\vee,R\right)\bigg)=(\theta^\One_{\pmb{d},n})$, an equality of divisors in $Z^1(R)$. Combining these observations, for every height one prime ideal $\p$ in $R$, we obtain the following short exact sequence of $R_\p$-modules:
\begin{align}  \label{eq:XdaggerP}
0 \rightarrow \XXX(\Q,D_{\pmb{d},n})^{\dagger} \otimes_R R_\p \rightarrow R_\p \rightarrow \frac{R_\p}{(\theta^\One_{\pmb{d},n})} \rightarrow 0
\end{align}

Since $R$ is integrally closed, equation (\ref{eq:XdaggerP}) lets us deduce that under the inclusion $\XXX(\Q,D_{\pmb{d},n})^{\dagger} \hookrightarrow R$ given in equation (\ref{eq:Xdagger}), we have a natural inclusion map $\XXX(\Q,D_{\pmb{d},n})^{\dagger} \hookrightarrow (\theta^\One_{\pmb{d},n})$, whose cokernel is pseudo-null. Since the ideal $\XXX(\Q,D_{\pmb{d},n})^{\dagger}$ is reflexive, this natural inclusion must be an equality. This shows that the $R$-module $\XXX(\Q,D_{\pmb{d},n})^{\dagger}$ is free.

The rest of the hypotheses listed in Theorem 4.3 in \cite{lei2018codimension} are part of the hypotheses of Theorem \ref{maintheorem}.

\begin{remark} \label{rem:two_unb}
Since it is important to our methods to construct the rank one module $\XXX(\Q,D_{\pmb{d},n})$, we do not consider the case when the two $p$-adic $L$-functions are both unbalanced, in the setting of the cyclotomic twist deformation of the tensor product of three Hida families. A similar construction of $\XXX(\Q,D_{\pmb{8},4})$ in these cases would produce an $R_{F,G,H}[[\Gamma_\Cyc]]$-module that has rank at least two.
\end{remark}

\section*{Acknowledgements}
We are extremely grateful to Henri Darmon, Ralph Greenberg and Ming-Lun Hsieh for answering many of our questions related to this project. We would like to thank the anonymous referees for having carefully read an earlier version of our manuscript as well as their valuable comments and suggestions.
\bibliographystyle{abbrv}
\bibliography{../biblio}

\begin{thebibliography}{10}

\bibitem{atiyah1969introduction}
M.~F. Atiyah and I.~G. Macdonald.
\newblock {\em Introduction to commutative algebra}.
\newblock Addison-Wesley Publishing Co., Reading, Mass.-London-Don Mills, Ont.,
  1969.

\bibitem{bleher2015higher}
F.~M. Bleher, T.~Chinburg, R.~Greenberg, M.~Kakde, G.~Pappas, R.~Sharifi, and
  M.~J. Taylor.
\newblock Higher {C}hern classes in {I}wasawa theory.
\newblock {\em Amer. J. Math.}, 142(2):627--682, 2020.

\bibitem{MR2290585}
S.~B\"ocherer and A.~A. Panchishkin.
\newblock Admissible {$p$}-adic measures attached to triple products of
  elliptic cusp forms.
\newblock {\em Doc. Math.}, (Extra Vol.):77--132, 2006.

\bibitem{MR1251956}
W.~Bruns and J.~Herzog.
\newblock {\em Cohen-{M}acaulay rings}, volume~39 of {\em Cambridge Studies in
  Advanced Mathematics}.
\newblock Cambridge University Press, Cambridge, 1993.

\bibitem{KazimOchiai}
K.~Buyukboduk and T.~Ochiai.
\newblock Main conjectures for higher rank nearly ordinary families -- {I},
  2017.
\newblock preprint, available at arXiv:1708.04494.

\bibitem{MR2148798}
J.~Coates and R.~Sujatha.
\newblock Fine {S}elmer groups of elliptic curves over {$p$}-adic {L}ie
  extensions.
\newblock {\em Math. Ann.}, 331(4):809--839, 2005.

\bibitem{MR3250064}
H.~Darmon and V.~Rotger.
\newblock Diagonal cycles and {E}uler systems {I}: {A} {$p$}-adic
  {G}ross-{Z}agier formula.
\newblock {\em Ann. Sci. \'Ec. Norm. Sup\'er. (4)}, 47(4):779--832, 2014.

\bibitem{eisenbud1995commutative}
D.~Eisenbud.
\newblock {\em Commutative algebra}, volume 150 of {\em Graduate Texts in
  Mathematics}.
\newblock Springer-Verlag, New York, 1995.
\newblock With a view toward algebraic geometry.

\bibitem{emerton2006variation}
M.~Emerton, R.~Pollack, and T.~Weston.
\newblock Variation of {I}wasawa invariants in {H}ida families.
\newblock {\em Invent. Math.}, 163(3):523--580, 2006.

\bibitem{greenberg2015triple}
M.~Greenberg and M.~A. Seveso.
\newblock Triple product {$p$}-adic {$L$}-functions for balanced weights.
\newblock {\em Math. Ann.}, 376(1-2):103--176, 2020.

\bibitem{greenberg1994iwasawa}
R.~Greenberg.
\newblock Iwasawa theory and {$p$}-adic deformations of motives.
\newblock In {\em Motives ({S}eattle, {WA}, 1991)}, volume~55 of {\em Proc.
  Sympos. Pure Math.}, pages 193--223. Amer. Math. Soc., Providence, RI, 1994.

\bibitem{MR1846466}
R.~Greenberg.
\newblock Iwasawa theory---past and present.
\newblock In {\em Class field theory---its centenary and prospect ({T}okyo,
  1998)}, volume~30 of {\em Adv. Stud. Pure Math.}, pages 335--385. Math. Soc.
  Japan, Tokyo, 2001.

\bibitem{MR2290593}
R.~Greenberg.
\newblock On the structure of certain {G}alois cohomology groups.
\newblock {\em Doc. Math.}, (Extra Vol.):335--391, 2006.

\bibitem{greenberg2010surjectivity}
R.~Greenberg.
\newblock Surjectivity of the global-to-local map defining a {S}elmer group.
\newblock {\em Kyoto J. Math.}, 50(4):853--888, 2010.

\bibitem{MR1835065}
M.~Harris and J.~Tilouine.
\newblock {$p$}-adic measures and square roots of special values of triple
  product {$L$}-functions.
\newblock {\em Math. Ann.}, 320(1):127--147, 2001.

\bibitem{hida1985ap}
H.~Hida.
\newblock A {$p$}-adic measure attached to the zeta functions associated with
  two elliptic modular forms. {I}.
\newblock {\em Invent. Math.}, 79(1):159--195, 1985.

\bibitem{hida1986galois}
H.~Hida.
\newblock Galois representations into {${\rm GL}_2({\bf Z}_p[[X]])$} attached
  to ordinary cusp forms.
\newblock {\em Invent. Math.}, 85(3):545--613, 1986.

\bibitem{MR868300}
H.~Hida.
\newblock Iwasawa modules attached to congruences of cusp forms.
\newblock {\em Ann. Sci. \'Ecole Norm. Sup. (4)}, 19(2):231--273, 1986.

\bibitem{hida1988p}
H.~Hida.
\newblock A {$p$}-adic measure attached to the zeta functions associated with
  two elliptic modular forms. {II}.
\newblock {\em Ann. Inst. Fourier (Grenoble)}, 38(3):1--83, 1988.

\bibitem{Hidafourier}
H.~Hida.
\newblock On {$p$}-adic {$L$}-functions of {${\rm GL}(2)\times {\rm GL}(2)$}
  over totally real fields.
\newblock {\em Ann. Inst. Fourier (Grenoble)}, 41(2):311--391, 1991.

\bibitem{hida1993elementary}
H.~Hida.
\newblock {\em Elementary theory of {$L$}-functions and {E}isenstein series},
  volume~26 of {\em London Mathematical Society Student Texts}.
\newblock Cambridge University Press, Cambridge, 1993.

\bibitem{hida1996search}
H.~Hida.
\newblock On the search of genuine {$p$}-adic modular {$L$}-functions for
  {${\rm GL}(n)$}.
\newblock {\em M\'em. Soc. Math. Fr. (N.S.)}, (67):vi+110, 1996.
\newblock With a correction to: ``On $p$-adic $L$-functions of
  ${\rm{G}}L(2){\times{G}}L(2)$ over totally real fields'' Ann. Inst. Fourier
  (Grenoble) {{\bf{4}}1} (1991), no. 2, 311--391.

\bibitem{hsieh2016hida}
M.-L. Hsieh.
\newblock Hida families and {$p$}-adic triple product {$L$}-functions.
\newblock {\em to appear in Amer. J. Math.}
\newblock available at arXiv:1705.02717.

\bibitem{hsiehyamana}
M.-L. Hsieh and S.~Yamana.
\newblock Four variable {$p$}-adic triple product {$L$}-functions.
\newblock preprint, available at arXiv:1906.10474.

\bibitem{KLZ}
G.~Kings, D.~Loeffler, and S.~L. Zerbes.
\newblock Rankin-{E}isenstein classes and explicit reciprocity laws.
\newblock {\em Camb. J. Math.}, 5(1):1--122, 2017.

\bibitem{MR1029028}
S.~Lang.
\newblock {\em Cyclotomic fields {I} and {II}}, volume 121 of {\em Graduate
  Texts in Mathematics}.
\newblock Springer-Verlag, New York, second edition, 1990.
\newblock With an appendix by Karl Rubin.

\bibitem{MR3381453}
A.~G.~B. Lauder.
\newblock Efficient computation of {R}ankin {$p$}-adic {$L$}-functions.
\newblock In {\em Computations with modular forms}, volume~6 of {\em Contrib.
  Math. Comput. Sci.}, pages 181--200. Springer, Cham, 2014.

\bibitem{lei2018codimension}
A.~Lei and B.~Palvannan.
\newblock Codimension two cycles in {I}wasawa theory and elliptic curves with
  supersingular reduction.
\newblock {\em Forum Math. Sigma}, 7:Paper No. e25, 81, 2019.

\bibitem{matsumura1989commutative}
H.~Matsumura.
\newblock {\em Commutative ring theory}, volume~8 of {\em Cambridge Studies in
  Advanced Mathematics}.
\newblock Cambridge University Press, Cambridge, second edition, 1989.
\newblock Translated from the Japanese by M. Reid.

\bibitem{neukirch2008cohomology}
J.~Neukirch, A.~Schmidt, and K.~Wingberg.
\newblock {\em Cohomology of number fields}, volume 323 of {\em Grundlehren der
  Mathematischen Wissenschaften [Fundamental Principles of Mathematical
  Sciences]}.
\newblock Springer-Verlag, Berlin, second edition, 2008.

\bibitem{palvannan2016algebraic}
B.~Palvannan.
\newblock {On {S}elmer groups and factoring {$\boldsymbol{p}$-adic
  $\boldsymbol{L}$-functions}}.
\newblock {\em International Mathematics Research Notices},
  2018(24):7483--7554, 05 2017.

\bibitem{palvannan2015homological}
B.~Palvannan.
\newblock Height one specializations of {S}elmer groups.
\newblock {\em Ann. Inst. Fourier (Grenoble)}, 69(1):303--334, 2019.

\bibitem{ribes2000profinite}
L.~Ribes and P.~Zalesskii.
\newblock {\em Profinite groups}, volume~40 of {\em Ergebnisse der Mathematik
  und ihrer Grenzgebiete. 3. Folge. A Series of Modern Surveys in Mathematics
  [Results in Mathematics and Related Areas. 3rd Series. A Series of Modern
  Surveys in Mathematics]}.
\newblock Springer-Verlag, Berlin, 2000.

\bibitem{saha2018algebraic}
J.~P. Saha.
\newblock An algebraic {$p$}-adic {$L$}-function for {$p$}-adic families.
\newblock {\em preprint, available at
  https://sites.google.com/site/jyotiprakashsaha/p}.

\bibitem{saha:tel-01124363}
J.~P. Saha.
\newblock {\em {An algebraic p-adic L-function for ordinary families}}.
\newblock Theses, {Universit{\'e} Paris Sud - Paris XI}, June 2014.

\bibitem{MR0238839}
W.~V. Vasconcelos.
\newblock On finitely generated flat modules.
\newblock {\em Trans. Amer. Math. Soc.}, 138:505--512, 1969.

\bibitem{MR0257062}
K.~Watanabe, T.~Ishikawa, S.~Tachibana, and K.~Otsuka.
\newblock On tensor products of {G}orenstein rings.
\newblock {\em J. Math. Kyoto Univ.}, 9:413--423, 1969.

\bibitem{MR3076731}
C.~A. Weibel.
\newblock {\em The {$K$}-book}, volume 145 of {\em Graduate Studies in
  Mathematics}.
\newblock American Mathematical Society, Providence, RI, 2013.
\newblock An introduction to algebraic $K$-theory.

\end{thebibliography}

\Addresses
\end{document}